\documentclass{amsart}
\usepackage{amsmath,amssymb}
\usepackage[mathcal]{euscript}

\newcommand{\bydef}{:=}

\newcommand{\vphi}{\varphi}
\newcommand{\eps}{\varepsilon}

\newcommand{\Sl}{\mathfrak{sl}}

\newcommand{\Psl}{\mathfrak{psl}}
\newcommand{\So}{\mathfrak{so}}
\newcommand{\Sp}{\mathfrak{sp}}

\newcommand{\NA}{N_A}
\newcommand{\AI}{\Gamma_A^\mathrm{(I)}}
\newcommand{\NAI}{N_A^\mathrm{(I)}}
\newcommand{\AII}{\Gamma_A^\mathrm{(II)}}
\newcommand{\NAII}{N_A^\mathrm{(II)}}

\newcommand{\C}{\Gamma_C}
\newcommand{\NC}{N_C}
\newcommand{\D}{\Gamma_D}
\newcommand{\ND}{N_D}
\newcommand{\M}{\Gamma_M}
\newcommand{\NM}{N_M}
\newcommand{\NLie}{N_X}

\newcommand{\NMA}{\hat N_M}
\newcommand{\NLieA}{\hat N_X}
\newcommand{\NAA}{\hat N_A}
\newcommand{\NCA}{\hat N_C}
\newcommand{\NDA}{\hat N_D}

\newcommand{\logsquaredterm}{\frac{2}{9\ln 2}(\ln n)^2}

\newcommand{\SP}{\mathrm{Sp}}
\newcommand{\ASP}{\mathrm{ASp}}

\DeclareMathOperator{\supp}{\mathrm{Supp}\,}
\DeclareMathOperator{\chr}{\mathrm{char}\,}
\newcommand{\FF}{\mathbb{F}}
\newcommand{\ZZ}{\mathbb{Z}}
\newcommand{\CC}{\mathbb{C}}
\newcommand{\cA}{\mathcal{A}}
\newcommand{\cB}{\mathcal{B}}

\newcommand{\parts}{P}
\newcommand{\fixed}{\mathrm{Fix}}
\newcommand{\partsmuq}[2]{P_{#1}(#2)}
\newcommand{\fits}[2]{f(#1,#2)}

\newtheorem{theorem}{Theorem}
\newtheorem{lemma}[theorem]{Lemma}

\theoremstyle{definition}
\newtheorem{df}{Definition}

\theoremstyle{remark}
\newtheorem{remark}{Remark}

\begin{document}

\author[M. Kochetov]{Mikhail Kochetov$^{\star}$}
\thanks{$^{\star}$Partially supported by the Natural Sciences and Engineering
Research Council (NSERC) of Canada, Discovery Grant \# 341792-07.}
\address{Department of Mathematics and Statistics,
Memorial University of Newfoundland, St. John's, NL, A1C5S7, Canada}
\email{mikhail@mun.ca}

\author[N. Parsons]{Nicholas Parsons$^{\dagger}$}
\thanks{$^{\dagger}$Supported by the NSERC Undergraduate Student Research
Award in Summer 2012 (under the supervision of M. Kochetov) with a financial contribution from Professor D. Summers' NSERC Discovery Grant.}
\address{Department of Mathematics and Statistics,
Memorial University of Newfoundland, St. John's, NL, A1C5S7, Canada}
\email{t29ngp@mun.ca}

\author[S. Sadov]{Sergey Sadov$^{\diamond}$}
\thanks{$^{\diamond}$Partially supported by the Natural Sciences and Engineering
Research Council (NSERC) of Canada, Discovery Grant \# 312573-05.}
\address{Department of Mathematics and Statistics,
Memorial University of Newfoundland, St. John's, NL, A1C5S7, Canada}
\email{sergey@mun.ca}

\title[Counting fine gradings]{Counting fine gradings on matrix algebras and on classical simple Lie algebras}



\begin{abstract}
Known classification results allow us to find the number of (equivalence classes of) fine gradings on matrix algebras and on classical simple Lie algebras over an algebraically closed field $\FF$ (assuming $\chr{\FF}\ne 2$ in the Lie case). The computation is easy for matrix algebras and especially for simple Lie algebras of type $B_r$ (the answer is just $r+1$), but involves counting orbits of certain finite groups in the case of Series $A$, $C$ and $D$. For $X\in\{A,C,D\}$, we determine the exact number of fine gradings, $N_X(r)$, on the simple Lie algebras of type $X_r$ with $r\le 100$ as well as the asymptotic behaviour of the average, $\hat N_X(r)$, for large $r$. In particular, we prove that there exist positive constants $b$ and $c$ such that $\exp(br^{2/3})\le\hat N_X(r)\le\exp(cr^{2/3})$. The analogous average for matrix algebras $M_n(\FF)$ is proved to be $a\ln n+O(1)$ where $a$ is an explicit constant depending on $\chr{\FF}$.
\end{abstract}

\maketitle


\section{Introduction}

Let $\cA$ be an algebra (not necessarily associative) over a field $\FF$ and let $G$ be a semigroup (written multiplicatively).

\begin{df}\label{df:G_graded_alg}
A {\em $G$-grading} on $\cA$ is a vector space decomposition
\[
\Gamma:\;\cA=\bigoplus_{g\in G} \cA_g
\]
such that
\[
\cA_g \cA_h\subset \cA_{gh}\quad\mbox{for all}\quad g,h\in G.
\]
If such a decomposition is fixed, $\cA$ is referred to as a {\em $G$-graded algebra}.
The {\em support} of $\Gamma$ is the set $\supp\Gamma\bydef\{g\in G\;|\;\cA_g\neq 0\}$.
\end{df}

The reader may consult the recent monograph \cite{EKmon} for background on gradings. In particular, there is more than one natural  equivalence relation on graded algebras, depending on whether or not it is desirable to fix $G$. In this paper we will use the following version, where $G$ is not fixed. 

\begin{df}\label{df:equ_grad}
Let $\Gamma:\, \cA=\bigoplus_{g\in G} \cA_g$ and $\Gamma':\,\cB=\bigoplus_{h\in H} \cB_h$ be two graded algebras, with supports $S$ and $T$, respectively.
We say that the graded algebras $\cA$ and $\cB$ (or the gradings $\Gamma$ and $\Gamma'$) are {\em equivalent} if there exists an isomorphism of algebras $\vphi\colon\cA\to\cB$ and a bijection $\alpha\colon S\to T$ such that $\varphi(\cA_s)=\cB_{\alpha(s)}$ for all $s\in S$.
\end{df}

It is known that if $\Gamma$ is a grading on a simple Lie algebra by any semigroup, then $\supp\Gamma$ generates an abelian group (see e.g. \cite{Ksur} or \cite[Proposition 1.12]{EKmon}). From now on, we will assume that all gradings are by {\em abelian groups}, which will be written additively. The cyclic group $\ZZ/m\ZZ$ will be denoted by $\ZZ_m$. We will also assume that the ground field $\FF$ is {\em algebraically closed}.

The so-called fine gradings on an algebra (defined below) are of special importance since they reveal the structure of the algebra and its automorphism group (if $\chr\FF=0$, then the fine gradings on a finite-dimensional algebra $\cA$ correspond to maximal quasitori in the automorphism group of $\cA$). 

\begin{df}\label{df:fine_grad}
Let $\Gamma:\,\cA=\bigoplus_{g\in G} \cA_g$ and $\Gamma':\,\cA=\bigoplus_{h\in H} \cA'_h$ be two gradings on the same algebra, with supports $S$ and $T$, respectively.
We will say that $\Gamma'$ is a {\em refinement} of $\Gamma$ (or $\Gamma$ is a {\em coarsening} of $\Gamma'$) if for any $t\in T$ there exists (unique) $s\in S$ such that $\cA'_t\subset\cA_s$. If, moreover, $\cA'_t\ne\cA_s$ for at least one $t\in T$, then the refinement is said to be {\em proper}. Finally, $\Gamma$ is said to be {\em fine} if it does not admit any proper refinements (in the class of gradings by abelian groups).
\end{df}

Gradings have recently been classified for many interesting algebras (see e.g. \cite{EKmon} and references therein). In particular, a classification of fine gradings up to equivalence is known for matrix algebras over an algebraically closed field $\FF$ of arbitrary characteristic \cite{HPP98,BSZ,BZ03} and for classical simple Lie algebras except $D_4$ in characteristic different from $2$ \cite{E09d,EK_Weyl2}. Type $D_4$ is different from all other members of Series $D$ due to the phenomenon of triality. In \cite{E09d}, fine gradings on the simple Lie algebra of type $D_4$ are classified in characteristic $0$; there are $17$ equivalence classes. As to the exceptional simple Lie algebras, fine gradings are classified for type $G_2$ ($\chr\FF\ne 2,3$) in  \cite{DM_g2,EK_g2f4}, for type $F_4$ ($\chr\FF\ne 2$) in \cite{DM_f4,EK_g2f4}, and for type $E_6$ ($\chr\FF=0$) in \cite{DV_e6}. The number of equivalence classes is, respectively, $2$, $4$ (only $3$ in the case $\chr\FF=3$), and $14$.

In the present paper, we are interested in the number of (equivalence classes of) fine gradings for matrix algebras and for classical simple Lie algebras of Series $A$, $C$ and $D$. It is easy to see that there are $2$ fine gradings on $\Sl_2(\FF)$ ($\chr{\FF}\ne 2$). The number of fine gradings on a few other members of these series over $\CC$ (and on their real forms) have been found in \cite{HPP98sl3,PPS01,PPS02,Svo08} using the description of maximal quasitori (``MAD subroups'') in \cite{HPP98}. The more recent classification results, as stated in \cite{EK_Weyl2} and \cite{EKmon}, reduce the problem to counting orbits of certain finite groups, so the number of fine gradings can be computed, in principle, for any member of these series over an algebraically closed field of characteristic different from $2$. Note that there is no work to be done for Series $B$ because there are exactly $r+1$ gradings on the simple Lie algebra of type $B_r$ (see \cite{E09d} or \cite[\S 3.4]{EKmon}).  We count the orbits using Burnside--Cauchy--Frobenius Lemma and the computer algebra system GAP (see \cite{GAP4}) to obtain the exact number of fine gradings for simple Lie algebras of types $A_r$, $C_r$ and $D_r$ up to $r=100$ (see Tables \ref{t:gradings_A}, \ref{t:gradings_C} and \ref{t:gradings_D}, respectively). The number of fine gradings on $M_n(\FF)$ is easily computed since it is expressed in terms of the partition function and the multiplicities of the prime factors of $n$. We state the answer for $n$ up to $100$ for completeness (see Table \ref{t:gradings_M}). The behaviour of the average number of fine gradings on $M_j(\FF)$ with $j\le n$ as $n\to\infty$ (Theorem \ref{thm:asymgradings_M}) is  derived from the known asymptotics of the number of abelian groups of order $\le n$ \cite{ErdSzek}. We also establish the asymptotic behaviour of the average number of fine gradings for simple Lie algebras of Series $A$ (Theorem \ref{thm:asymgradings_A}), $C$ (Theorem \ref{thm:asymgradings_C}) and $D$ (Theorem \ref{thm:asymgradings_D}); this number exhibits {\em intermediate growth}: faster than any polynomial but slower than any exponential. The proof of these results is based on the asymptotic analysis of certain binomial coefficients (Section \ref{sec:asymACD}).


\section{Matrix Algebras}
\label{sec:M}

Fine gradings on $M_n(\CC)$ were classified in \cite{HPP98} in terms of the corresponding ``MAD subgroups'' of $\mathrm{PGL_n(\CC)}$. The approach in \cite{BSZ} was to look directly at the structure of the graded algebra, which allowed a  generalization to $M_n(\FF)$ over any algebraically closed field \cite{BZ03}. We state the classification using the notation of \cite[Corollary 2.6]{EK_Weyl1} and \cite[\S 2.3]{EKmon}. We do not explain this notation here, as our present concern is the number of gradings and not their explicit form. The subscript $M$ stands for matrices; later we will use subscripts $A$, $C$ and $D$ for the corresponding series of classical Lie algebras.

\begin{theorem}[\cite{HPP98,BZ03}]\label{matr_fine}
Let $\Gamma$ be a fine abelian group grading on the matrix algebra $M_n(\FF)$ over an algebraically closed field $\FF$. Then $\Gamma$ is equivalent to some $\M(T,k)$ where $T$ is a finite abelian group of the form $\ZZ_{\ell_1}^2\times\cdots\times\ZZ_{\ell_r}^2$ (i.e., a Cartesian square), $\chr\FF$ does not divide $|T|$, and  $k\ell_1\cdots\ell_r=n$. Two gradings $\M(T_1,k_1)$ and $\M(T_2,k_2)$ are equivalent if and only if $T_1\cong T_2$ and  $k_1=k_2$.\hfill{$\square$}
\end{theorem}

It follows that the number of fine gradings on $M_n(\FF)$ is given by
\begin{equation}\label{eq:num_grad_M_1}
\NM(n)=\sum_{\ell\,|\,n}N_{ab}(\ell),
\end{equation}
where $N_{ab}(\ell)$ is the number of (isomorphism classes of) abelian groups of order $\ell$ (Online Encyclopaedia of Integer Sequences A000688) and, if $\chr{\FF}=p$, the summation is restricted to $\ell$ that are not divisible by $p$.

\subsection{Counting gradings}

Factoring $\ell=p_1^{m_1}\cdots p_s^{m_s}$ in equation \eqref{eq:num_grad_M_1}, where $p_i\ne\chr{\FF}$ are distinct primes, we obtain  $N_{ab}(\ell)=\parts(m_1)\cdots\parts(m_s)$ where $\parts(m)$ denotes the number of partitions of a non-negative integer $m$ (with the convention $\parts(0)=1$). Hence, if $n=p_1^{\alpha_1}\cdots p^{\alpha_s}_s$ and $\chr{\FF}=0$ or if $n=p_1^{\alpha_1}\cdots p^{\alpha_s}_s p^\alpha$ and $\chr{\FF}=p$ then equation \eqref{eq:num_grad_M_1} can be rewritten as
\begin{equation}\label{eq:num_grad_M_2}
\NM(n)=\prod_{i=1}^s\sum_{j=0}^{\alpha_i}\parts(j).
\end{equation}
Table \ref{t:gradings_M} displays the numbers $\NM(n)$ for $n\le 100$ in the case $\chr{\FF}=0$, which were calculated using equation \eqref{eq:num_grad_M_2}.

\begin{table}[h]
\begin{tabular}{ccccc}
\begin{tabular}{|c|c|}
\hline
$n$ & $\NM(n)$\\ 
\hline
1 & 1\\
\hline
2 & 2\\
\hline
3 & 2\\
\hline
4 & 4\\
\hline
5 & 2\\
\hline
6 & 4\\
\hline
7 & 2\\
\hline
8 & 7\\
\hline
9 & 4\\
\hline
10 & 4\\
\hline
11 & 2\\
\hline
12 & 8\\
\hline
13 & 2\\
\hline
14 & 4\\
\hline
15 & 4\\
\hline
16 & 12\\
\hline
17 & 2\\
\hline
18 & 8\\
\hline
19 & 2\\
\hline
20 & 8\\
\hline
\end{tabular}
&
\begin{tabular}{|c|c|}
\hline
$n$ & $\NM(n)$\\ 
\hline
21 & 4\\
\hline
22 & 4\\
\hline
23 & 2\\
\hline
24 & 14\\
\hline
25 & 4\\
\hline
26 & 4\\
\hline
27 & 7\\
\hline
28 & 8\\
\hline
29 & 2\\
\hline
30 & 8\\
\hline
31 & 2\\
\hline
32 & 19\\
\hline
33 & 4\\
\hline
34 & 4\\
\hline
35 & 4\\
\hline
36 & 16\\
\hline
37 & 2\\
\hline
38 & 4\\
\hline
39 & 4\\
\hline
40 & 14\\
\hline
\end{tabular}
&
\begin{tabular}{|c|c|}
\hline
$n$ & $\NM(n)$\\
\hline
41 & 2\\
\hline
42 & 8\\
\hline
43 & 2\\
\hline
44 & 8\\
\hline
45 & 8\\
\hline
46 & 4\\
\hline
47 & 2\\
\hline
48 & 24\\
\hline
49 & 4\\
\hline
50 & 8\\
\hline
51 & 4\\
\hline
52 & 8\\
\hline
53 & 2\\
\hline
54 & 14\\
\hline
55 & 4\\
\hline
56 & 14\\
\hline
57 & 4\\
\hline
58 & 4\\
\hline
59 & 2\\
\hline
60 & 16\\
\hline
\end{tabular}
&
\begin{tabular}{|c|c|}
\hline
$n$ & $\NM(n)$\\
\hline
61 & 2\\
\hline
62 & 4\\
\hline
63 & 8\\
\hline
64 & 30\\
\hline
65 & 4\\
\hline
66 & 8\\
\hline
67 & 2\\
\hline
68 & 8\\
\hline
69 & 4\\
\hline
70 & 8\\
\hline
71 & 2\\
\hline
72 & 28\\
\hline
73 & 2\\
\hline
74 & 4\\
\hline
75 & 8\\
\hline
76 & 8\\
\hline
77 & 4\\
\hline
78 & 8\\
\hline
79 & 2\\
\hline
80 & 24\\
\hline
\end{tabular}
&
\begin{tabular}{|c|c|}
\hline
$n$ & $\NM(n)$\\
\hline
81 & 12\\
\hline
82 & 4\\
\hline
83 & 2\\
\hline
84 & 16\\
\hline
85 & 4\\
\hline
86 & 4\\
\hline
87 & 4\\
\hline
88 & 14\\
\hline
89 & 2\\
\hline
90 & 16\\
\hline
91 & 4\\
\hline
92 & 8\\
\hline
93 & 4\\
\hline
94 & 4\\
\hline
95 & 4\\
\hline
96 & 38\\
\hline
97 & 2\\
\hline
98 & 8\\
\hline
99 & 8\\
\hline
100 & 16\\
\hline
\end{tabular}
\end{tabular}

\medskip

\caption{Number of fine gradings, $\NM(n)$, on the matrix algebra $M_n(\FF)$, where $\chr{\FF}=0$.}
\label{t:gradings_M}
\end{table}

\subsection{Asymptotic behaviour}

The function $\NM(n)$ behaves irregularly, so we introduce the following:
\begin{equation}\label{defNMA}
 \NMA(n)=\frac{1}{n}\sum_{j=1}^n \NM(j),
\end{equation}
which is the average number of fine gradings on the algebras $M_j(\FF)$ with $j\leq n$.

\begin{theorem}\label{thm:asymgradings_M}
Let $\FF$ be an algebraically closed field of characteristic $c$. The following asymptotic formula holds:
\begin{equation}\label{asymgradings_M0}
\NMA(n)=a_c \ln n +O(1).
\end{equation}

In the case $c=0$, the constant is
\begin{equation}\label{consta0}
 a_0=\prod_{m=2}^\infty \zeta(m)\approx 2.2948566;
\end{equation}
here $\zeta(\cdot)$ is the Riemann zeta function.

In the case $c>0$, the constant is
\begin{equation}\label{constac}
 a_c=a_0\,\prod_{m=2}^\infty (1-c^{-m}).
\end{equation}
\end{theorem}

The infinite product appearing in the formula for $a_c$ is a particular case of {\it $q$-Pochhammer symbol} (with $q=1/c$)
known in the theory of elliptic functions, theory of partitions, and elsewhere. In standard notation, it is abbreviated as
\[
\prod_{m=2}^\infty (1-c^{-m}) = (c^{-2}; c^{-1})_\infty.
\]
The numerical values of $(c^{-2};c^{-1})_\infty$ for prime $c\le 13$ and the corresponding values of the constants $a_c$ are given in Table \ref{t:constants}.

\begin{table}[h]
\begin{tabular}{|c|c|c|c|c|c|c|}
\hline
$c$ & 2 & 3 & 5 & 7 & 11 & 13\\
\hline
$(c^{-2};c^{-1})_\infty$ & 0.577576 & 0.840189 & 0.950416 & 0.976261 & 0.990916 & 0.993593\\
\hline
$a_c$ & 1.325455 & 1.928114 & 2.181068 & 2.240380 & 2.274010 & 2.280153\\
\hline
\end{tabular}

\medskip

\caption{Numerical values of $(c^{-2};c^{-1})_\infty$ and $a_c$.}
\label{t:constants}
\end{table}

\begin{proof}
(a) Case $\chr\FF=0$. We refer to the classical result \cite{ErdSzek} on the average number of abelian groups of order $j$ with $j\leq n$:
\begin{equation}\label{asymabgrp0}
\frac{1}{n}\sum_{j=1}^n N_{ab}(j)= a_0 +O(n^{-1/2}).
\end{equation}
For simplicity denote the sum on the left side of \eqref{asymabgrp0} by $F(n)$ and also set $F(0)=0$ by definition.
Then
\[
\begin{split}
  n\NMA(n) &=
\sum_{j\leq n}\;\sum_{\ell\,|\,j} N_{ab}(\ell)=
\sum_{j=1}^n N_{ab}(j) \left\lfloor\frac{n}{j}\right\rfloor
\\[2ex]  &=
\sum_{j=1}^n (F(j)-F(j-1)) \frac{n}{j} - \sum_{j=1}^n N_{ab}(j) \left(\frac{n}{j}-\left\lfloor\frac{n}{j}\right\rfloor\right).
\end{split}
\]
The second sum has positive terms and is majorized by $\sum_{j=1}^n N_{ab}(j)=O(n)$.
Now,
\[
 \sum_{j=1}^n (F(j)-F(j-1)) \frac{n}{j} = F(n)+n\sum_{j=1}^{n-1} \frac{F(j)}{j(j+1)};
\]
here $F(n)=O(n)$, the sum
\[
 \sum_{j=1}^{n-1} \frac{F(j)-a_0 j}{j(j+1)}
\]
is $O(1)$ since its terms are $O(j^{-3/2})$ by \eqref{asymabgrp0},
while
\[
 \sum_{j=1}^n\frac{a_0 j}{j(j+1)}= a_0\ln n+ O(1).
\]
This completes the proof for $\chr\FF=0$.

\medskip
(b) Case $c=\chr\FF>0$.
The summatory function $n\NMA(n)$ in this case is
\[
 n\NMA(n)=\sum_{j\leq n}\;\sum_{\ell\,|\,j,\;c\,\nmid\,\ell} N_{ab}(\ell) =
\sum_{\ell=1}^n f(\ell) \left\lfloor\frac{n}{\ell}\right\rfloor,
\]
where we set
\[
 f(n)=\left\{ \begin{array}{ll}
   N_{ab}(n) & \mbox{if}\; c\nmid n, \\
   0         & \mbox{if}\; c\,|\,n.
\end{array}
\right.
\]
Similarly to (a), the formula \eqref{asymgradings_M0} with constant \eqref{constac} will follow
from the asymptotics of the summatory function $F(n)= \sum_{j=1}^n f(j)$
\begin{equation}\label{asymabgrpc}
F(n) = a_c n +O(n^{1/2}).
\end{equation}

A proof of \eqref{asymabgrpc} is a slight variation of the proof of \eqref{asymabgrp0} given in \cite{ErdSzek}.
The argument goes through with only one change: the prime $c$ is not participating in any products.
This results in the constant $A_1$ of \cite{ErdSzek}, which equals our $a_0$, being replaced by
\[
a_c= \prod_{m=2}^\infty \prod_{p\neq c} (1+p^{-m}+p^{-2m}+\cdots),
\]
where the inner product runs over all primes $p\neq c$. Thus
\[
 a_c=\prod_{m=2}^\infty \zeta(m)(1-c^{-m}) = a_0\, \prod_{m=2}^\infty (1-c^{-m}).
\]
\end{proof}

\begin{remark}
The asymptotic formula \eqref{asymabgrp0} has been significantly refined by many later authors, see e.g.
\cite{Liu,HB}. There is little doubt that formula \eqref{asymgradings_M0} can be refined similarly, but this endeavour is beyond the scope of the current paper. 
\end{remark}


\section{Lie Algebras of Series $A$}
\label{sec:A}

The classification of fine gradings on all simple Lie algebras of Series $A$ was established in \cite{E09d} for the case $\chr{\FF}=0$. We state the result in purely combinatorial terms, as it appears in \cite{EK_Weyl2} and \cite[\S 3.3]{EKmon}. There are two types of gradings, which we distinguish using superscripts (I) and (II). The subscript $A$ refers to the series of Lie algebras. We do not introduce the gradings explicitly because we are only interested in their number. A {\em multiset} in a set $X$ is a function $X\to\ZZ_{\ge 0}$ that assigns to each point its multiplicity. If a group $G$ acts on $X$, then it also acts on the multisets in $X$. The relevant group here is $\ASP_{2m}(2)$, the semidirect product $\ZZ_2^{2m}\rtimes\SP_{2m}(2)$ of the symplectic group $\SP_{2m}(2)$ and the vector group $\ZZ_2^{2m}$. Each  element $(t,A)\in\ZZ_2^{2m}\rtimes\SP_{2m}(2)$ acts on $\ZZ_2^{2m}$ in the natural way: $x\mapsto Ax+t$.

\begin{theorem}[\cite{E09d,EK_Weyl2}]\label{A_fine}
Let $\FF$ be an algebraically closed field, $\chr{\FF}\neq 2$. Let $n\ge 3$ if $\chr{\FF}\ne 3$ and $n\ge 4$ if $\chr{\FF}=3$. Then any fine grading on $\Psl_n(\FF)$ is equivalent to one of the following:
\begin{itemize}
\item $\AI(T,k)$ where $T$ is as in Theorem \ref{matr_fine}, $k$ is a positive integer, $k\sqrt{|T|}=n$, and $k\ge 3$ if $T$
is an elementary $2$-group;
\item $\AII(T,q,s,\tau)$ where $T$ is an elementary $2$-group of even
rank, $q$ and $s$ are non-negative integers, $(q+2s)\sqrt{|T|}=n$,
$\tau=(t_1,\ldots,t_q)$ is a $q$-tuple of elements of $T$, and $t_1\ne
t_2$ if $q=2$ and $s=0$.
\end{itemize}
Gradings belonging to different types listed above are not equivalent.
Within each type, we have the following:
\begin{itemize}
\item $\AI(T_1,k_1)$ and $\AI(T_2,k_2)$ are equivalent if and only if\\
$T_1\cong T_2$ and $k_1=k_2$;
\item $\AII(T_1,q_1,s_1,\tau_1)$ and $\AII{}(T_2,q_2,s_2,\tau_2)$ are
equivalent if and only if\\ $T_1\cong T_2$, $q_1=q_2$, $s_1=s_2$ and,
identifying $T_1=T_2=\ZZ_2^{2m}$, $\Sigma(\tau_1)$ is conjugate to
$\Sigma(\tau_2)$ by the natural action of $\ASP_{2m}(2)$,
where $\Sigma(\tau)$ is the multiset for which the multiplicity of any point $t$ is the number of times $t$ appears among the components of the $q$-tuple $\tau$.\qed
\end{itemize}
\end{theorem}

Since, for Series $A$, the rank $r$ is related to the matrix size $n$ as $n=r+1$, Theorem \ref{A_fine} implies that the number of fine gradings of Type I on the simple Lie algebra $A_r$ ($r\ge 2$) is given by
\begin{equation}\label{eq:num_grad_AI}
\NAI(r)=\left\{\begin{array}{ll}
               \NM(r+1)-2 &\mbox{if $r+1$ is a power of 2},\\
               \NM(r+1)   &\mbox{otherwise},
               \end{array}\right.
\end{equation}
where $\NM(n)$ is the number of fine gradings on $M_n(\FF)$, which is discussed in the previous section. 

In order to calculate the number of fine gradings of Type II, we need to determine the number of orbits, $N(m,q)$,
of $\ASP_{2m}(2)$ on multisets of size $q$ in $T=\ZZ_2^{2m}$. Clearly, $N(m,q)$ does not exceed the total number of such multisets. Note that any multiset of size $q$ determines a partition of the integer $q$ by looking at the (nonzero) multiplicities and forgetting to which points in $T$ they belong. Clearly, any permutation of $T$ leaves invariant the set of multisets belonging to a fixed partition. Hence, we obtain bounds:
\begin{equation}\label{eq:triv_bounds}
\parts_{2^{2m}}(q)\le N(m,q)\le\binom{q+2^{2m}-1}{q},
\end{equation}
where $\parts_M(q)$ is the number of partitions of $q$ into at most $M$ positive parts. Note that $\ASP_2(2)$ is the full group of permutations on $\ZZ_2^2$, hence the lower bound is achieved if $m=1$. It is also achieved if $q\le 2$ because $\SP_{2m}(2)$ acts irreducibly and hence $\ASP_{2m}(2)$ acts $2$-transitively on $\ZZ_2^{2m}$.

Another lower bound for $N(m,q)$, which is better than that in \eqref{eq:triv_bounds} for any fixed $m>1$ and sufficiently large $q$, comes from the obvious fact that the size of a $G$-orbit cannot exceed the size of $G$. Thus
\begin{equation}\label{eq:lower_bound_G}
\frac{1}{|G_m|}\binom{q+2^{2m}-1}{q}\le N(m,q),
\end{equation}
where $G_m=\ASP_{2m}(2)$. This bound is indeed better as $q\to\infty$, because $\binom{q+M-1}{q}\sim\frac{q^{M-1}}{(M-1)!}$,  $G_m$ is not the full group of permutations if $m>1$, and it is known that $P_M(q)\sim \frac{1}{M!}\frac{q^{M-1}}{(M-1)!}$. The upper bound in \eqref{eq:triv_bounds} and the lower bound \eqref{eq:lower_bound_G} will be used
to obtain asymptotic results in Section~\ref{sec:asymACD}.

\subsection{Counting orbits}

We start with a few general remarks. It is customary to write partitions as decreasing sequences of positive integers: $\kappa=(k_1,\ldots,k_\ell)$, where $k_1\ge\ldots\ge k_\ell$ and $\ell=\ell(q)$ is called the {\em length} of $\kappa$. The notation $\kappa\vdash q$ means $\sum_j k_j=q$. We can also write $\kappa=(q_1^{(\ell_1)},\ldots,q_s^{(\ell_s)})$ where $q_1>\ldots>q_s$ and the superscript $(\ell_j)$ indicates the number of times the value $q_j$ is repeated; $\sum_j\ell_j=\ell$. For example, $(4,4,4,3,1)$ can be written as $(4^{(3)},3^{(1)},1^{(1)})$ or just $(4^{(3)},3,1)$. When working with partitions of length $\le M$ for a fixed $M$, it is sometimes convenient to append zeros at the end so the total number of parts is formally $M$. For example, with $M=7$, the partition $(4^{(3)},3,1)$ may be written as $(4^{(3)},3,1,0^{(2)})$.

Let $T$ be a set of $M$ elements. As pointed out above, any multiset of size $q$ in $T$ determines a partition $\kappa\vdash q$ of length $\le M$, and the set of all multisets belonging to a fixed $\kappa\vdash q$ is invariant under any group $G$ acting on $T$. For $T=\ZZ_2^{2m}$ and $G=G_m$, let $N(m,\kappa)$ be the number of orbits in this set. Thus 
\begin{equation}\label{eq:Npartitions}
N(m,q)=\sum_{\kappa\vdash q,\,\ell(\kappa)\le 2^{2m}}N(m,\kappa).
\end{equation}

Similarly, a partition $\kappa=(q_1^{(\ell_1)},\ldots,q_s^{(\ell_s)})$ of length $M$ determines a partition of $M$ given by sorting the sequence $(\ell_1,\ldots,\ell_s)$. If $\kappa$ has length $\ell<M$, we write $\kappa=(q_1^{(\ell_1)},\ldots,q_s^{(\ell_s)},0^{(M-\ell)})$ and sort the sequence $(\ell_1,\ldots,\ell_s,M-\ell)$. In other words, we regard $\kappa$ as a multiset of size $M$ in $\ZZ_{\ge 0}$ and assign to it a partition of $M$ as was done before to multisets in $T$. For example, with $M=7$, the partition $(4,4,4,3,1)$ gives $(3,2,1,1)$. 

If $\kappa=(k_1,\ldots,k_M)$ and $\kappa'=(k'_1,\ldots,k'_M)$ are partitions of length $\le M$ that determine the same partition of $M$, then there is a bijection $\vphi\colon\ZZ_{\ge 0}\to\ZZ_{\ge 0}$ such that $k'_j=\vphi(k_j)$ for all $j$. If $T$ is a set of $M$ elements, then post-composition with any such $\vphi$ defines a bijection from the multisets in $T$ belonging to $\kappa$ to those belonging to $\kappa'$. Clearly, this bijection is $G$-equivariant for any group $G$ acting on $T$. In particular, we have $N(m,\kappa)=N(m,\kappa')$.

The above remarks show that, for a fixed $m$, the structure of $G_m$-orbits on multisets of all sizes in $T=\ZZ_2^{2m}$ is determined by the orbits belonging to a finite number of partitions. However, this number is very large: $\parts(M)\sim\frac{1}{4\sqrt{3}M}\exp\Big(\pi\sqrt{\frac{2M}{3}}\Big)$, $M=2^{2m}$.

For small values of $m$ and $q$, one can find the orbits on multisets of size $q$ using a standard function of GAP. (To speed up the calculation, one can work separately with multisets belonging to each partition $\kappa\vdash q$.) If we do not need representatives of the orbits (say, in order to construct explicitly all fine gradings of Type II on a simple Lie algebra of Series $A$) and just want to count their number then we can use the following well-known fact.

\begin{lemma}[Burnside--Cauchy--Frobenius]\label{BCF}
Let $G$ be a finite group acting on a finite set $X$. Then the number of orbits equals the average number of fixed points:
\[
|X/G|=\frac{1}{|G|}\sum_{g\in G}\fixed(g),
\]
where $\fixed(g)$ is the number of points in $X$ fixed by $g$.\qed
\end{lemma}

An important addition to this lemma is the observation that if $g$ and $g'$ are conjugate, then $\fixed(g)=\fixed(g')$, so
\begin{equation}\label{eq:BCF}
|X/G|=\frac{1}{|G|}\sum_i c_i\fixed(g_i),
\end{equation}
where the summation is over the conjugacy classes of $G$, with $g_i$ and $c_i$ being a representative and the size of the $i$-th class. The number of conjugacy classes in $G_m$ is small compared to $|G_m|$. 

We also used the following observation to make the calculation of the number of fixed points more efficient. Let $g$ be a permutation on a set $T$ of size $M$ and let $\lambda=(\lambda_1,\ldots,\lambda_\ell)\vdash M$ be the cycle structure of $g$, i.e., the $\lambda_j$ are the cycle lenghs of $g$ (including the trivial cycles). Let $\kappa$ be a partition of length $\le M$ and let $\mu=(\mu_1,\ldots,\mu_s)$ be the corresponding partition of $M$. Then the number of fixed points of $g$ among the multisets in $T$ belonging to $\kappa$ is equal to the number of functions $\vphi\colon\{1,\ldots,\ell\}\to\{1,\ldots,s\}$ such that $\sum_{j\,:\,\vphi(j)=k}\lambda_j=\mu_k$ for all $k=1,\ldots,s$. We denote this number by $\fits{\lambda}{\mu}$. Informally, $\fits{\lambda}{\mu}$ is the number of ways to fit $\ell$ pigeons of volumes given by $\lambda$ into $s$ holes of capacities given by $\mu$ so that each hole is filled to its full capacity. For example, for $\lambda=(1^{(M)})$ and $\mu=(M)$ we have $\fits{\lambda}{\mu}=1$; for $\lambda=(M)$ and $\mu=(1^{(M)})$ we have $\fits{\lambda}{\mu}=0$; for $\lambda=(3,3,2)$ and $\mu=(5,3)$ we have $\fits{\lambda}{\mu}=2$. (Pigeons are distinguishable even if they have the same volume; holes are distinguishable even if they have the same capacity.) 

Hence equation \eqref{eq:BCF} can be rewritten as
\begin{equation}\label{eq:N_BCFkappa}
N(m,\kappa)=\frac{1}{|G_m|}\sum_{\lambda\vdash 2^{2m}}c(\lambda)\fits{\lambda}{\mu},
\end{equation}
where $c(\lambda)$ is the sum of $c_i$ over all $i$ such that $g_i$ has cycle structure $\lambda$. Finally, let $\partsmuq{\mu}{q}$ be the number of partitions $\kappa\vdash q$ of length $\le M$ such that the corresponding partition of $M$ is $\mu$. Then, in view of \eqref{eq:N_BCFkappa}, we can rewrite \eqref{eq:Npartitions} as follows:
\begin{equation}\label{eq:N_BCFlambdamu}
N(m,q)=\frac{1}{|G_m|}\sum_{\lambda,\mu\vdash 2^{2m}}c(\lambda)\fits{\lambda}{\mu}\partsmuq{\mu}{q}.
\end{equation}
Note that the number of partitions $\lambda$ that actualy occur in the sum is at most the number of conjugacy classes of $G_m$; the number of partitions $\mu$ that actually occur is bounded by $\parts_{2^{2m}}(q)\le\parts(q)$. 

In GAP, we defined $G_m$ by converting  $\SP_{2m}(2)$ to a permutation group on $M=2^{2m}$ points and adding one generator for the translation by a nonzero vector. Then we obtained representatives and sizes of conjugacy classes using a standard function and found the numbers $N(m,q)$ using \eqref{eq:N_BCFlambdamu}. Some of these numbers are shown in Table \ref{t:orbitsASp}.

\begin{remark}
Alternatively, one can use a reformulation of \eqref{eq:N_BCFlambdamu} in terms of generating functions afforded by P\'olya's Theorem. Namely, the generating function $g_m(t)=1+\sum_{q=1}^\infty N(m,q)t^q$ is given by
\[
g_m(t)=Z_m\Big(\frac{1}{1-t},\frac{1}{1-t^2},\frac{1}{1-t^3},\ldots\Big)
\]
where $Z_m(x_1,x_2,x_3,\ldots)$ is the {\em cycle index} of $G_m$, i.e., the sum of the terms 
\[
\frac{c(\lambda)}{|G_m|}x_1^{k_1}\cdots x_M^{k_M}
\] 
where $k_i$ is the number of parts of $\lambda\vdash M$ that are equal to $i$. 
The numbers $N(m,q)$ can be obtained using a computer algebra system to expand $g_m(t)$ into a power series at $t=0$.
\end{remark}

\begin{table}[h]
\begin{tabular}{|c|c|c|c|c|c|c|c|c|c|c|c|c|}
\hline
$q$ & 1 & 2 & 3 & 4 & 5 & 6 & 7 & 8 & 9 & 10 & 11 & 12\\
\hline
$N(1,q)$ & 1 & 2 & 3 & 5 & 6 & 9 & 11 & 15 & 18 & 23 & 27 & 34\\
\hline
$N(2,q)$ & 1 & 2 & 4 & 9 & 17 & 38 & 74 &158 & 318 & 657 & 1304 & 2612\\
\hline
$N(3,q)$ & 1 & 2 & 4 & 10 & 22 & 67 & 202 & 755 & 3082 & 14493 & 72584 & 379501\\
\hline
$N(4,q)$ & 1 & 2 & 4 & 10 & 23 & 75 & 265 & 1352 & 9432 & 98773 & 1398351 & 23613147\\
\hline
$N(5,q)$ & 1 & 2 & 4 & 10 & 23 & 76 & 275 & 1495 & 12196 & 183053 & 5075226 & 226160064\\
\hline
\end{tabular}

\medskip

\caption{Number of orbits, $N(m,q)$, of $\ASP_{2m}(2)$ on the set of multisets of size $q$ in $T=\ZZ_2^{2m}$.}
\label{t:orbitsASp}
\end{table}

\subsection{Counting gradings}

The number of fine gradings of Type I, $\NAI(r)$, is given by equation \eqref{eq:num_grad_AI}, see also Table \ref{t:gradings_M}. As to Type II gradings, we will use the following notation. For $n=2^\alpha k$ where $k$ is odd, set
\begin{equation}\label{eq:f_A}
f_0(n)=\sum_{m=0}^\alpha\sum_{s=0}^{\lfloor 2^{\alpha-m-1}k\rfloor}N(m,2^{\alpha-m}k-2s),
\end{equation}
with the convention $N(0,q)=1$ for all $q$. Then Theorem \ref{A_fine} implies that the number of fine gradings of Type II, $\NAII(r)$, is given by
\begin{equation}\label{eq:num_grad_AII}
\NAII(r)=\left\{\begin{array}{ll}
                f_0(r+1)-1 &\mbox{if $r+1$ is a power of 2},\\
                f_0(r+1)   &\mbox{otherwise}.
                \end{array}\right.
\end{equation}
The total number of fine gradings on $A_r$ is, of course, $\NA(r)=\NAI(r)+\NAII(r)$. We calculated these numbers for $r\le 100$ assuming $\chr{\FF}=0$. The results are stated in Table \ref{t:gradings_A}. Note that if $r$ is even then $\NAII(r)=\frac{r}{2}+1$.

\begin{remark}
If $\chr{\FF}=p>0$ ($p\ne 2$), then the numbers $\NAI(r)$ must be adjusted to exclude the prime factor $p$ when it divides $r+1$. The case $A_2$ requires special treatment if $\chr{\FF}=3$; the number of fine gradings turns out to be $2$ instead of $3$.
\end{remark}

\begin{table}[h]
\begin{tabular}{ccc}
\begin{tabular}{|c|c|c|c|}
\hline
$r$ & (I) & (II) & $\NA(r)$ \\
\hline
2 & 2 & 2 & 4\\
\hline
3 & 2 & 6 & 8\\
\hline
4 & 2 & 3 & 5\\
\hline
5 & 4 & 8 & 12\\
\hline
6 & 2 & 4 & 6\\
\hline
7 & 5 & 16 & 21\\
\hline
8 & 4 & 5 & 9\\
\hline
9 & 4 & 16 & 20\\
\hline
10 & 2 & 6 & 8\\
\hline
11 & 8 & 29 & 37\\
\hline
12 & 2 & 7 & 9\\
\hline
13 & 4 & 29 & 33\\
\hline
14 & 4 & 8 & 12\\
\hline
15 & 10 & 56 & 66\\
\hline
16 & 2 & 9 & 11\\
\hline
17 & 8 & 49 & 57\\
\hline
18 & 2 & 10 & 12\\
\hline
19 & 8 & 88 & 96\\
\hline
20 & 4 & 11 & 15\\
\hline
21 & 4 & 78 & 82\\
\hline
22 & 2 & 12 & 14\\
\hline
23 & 14 & 157 & 171\\
\hline
24 & 4 & 13 & 17\\
\hline
25 & 4 & 119 & 123\\
\hline
26 & 7 & 14 & 21\\
\hline
27 & 8 & 247 & 255\\
\hline
28 & 2 & 15 & 17\\
\hline
29 & 8 & 175 & 183\\
\hline
30 & 2 & 16 & 18\\
\hline
31 & 17 & 441 & 458\\
\hline
32 & 4 & 17 & 21\\
\hline
33 & 4 & 249 & 253\\
\hline
34 & 4 & 18 & 22\\
\hline
\end{tabular}
&
\begin{tabular}{|c|c|c|c|}
\hline
$r$ & (I) & (II) & $\NA(r)$ \\
\hline
35 & 16 & 717 & 733\\
\hline
36 & 2 & 19 & 21\\
\hline
37 & 4 & 345 & 349\\
\hline
38 & 4 & 20 & 24\\
\hline
39 & 14 & 1305 & 1319\\
\hline
40 & 2 & 21 & 23\\
\hline
41 & 8 & 467 & 475\\
\hline
42 & 2 & 22 & 24\\
\hline
43 & 8 & 2269 & 2277\\
\hline
44 & 8 & 23 & 31\\
\hline
45 & 4 & 619 & 623\\
\hline
46 & 2 & 24 & 26\\
\hline
47 & 24 & 4284 & 4308\\
\hline
48 & 4 & 25 & 29\\
\hline
49 & 8 & 806 & 814\\
\hline
50 & 4 & 26 & 30\\
\hline
51 & 8 & 7700 & 7708\\
\hline
52 & 2 & 27 & 29\\
\hline
53 & 14 & 1033 & 1047\\
\hline
54 & 4 & 28 & 32\\
\hline
55 & 14 & 14592 & 14606\\
\hline
56 & 4 & 29 & 33\\
\hline
57 & 4 & 1305 & 1309\\
\hline
58 & 2 & 30 & 32\\
\hline
59 & 16 & 26426 & 26442\\
\hline
60 & 2 & 31 & 33\\
\hline
61 & 4 & 1628 & 1632\\
\hline
62 & 8 & 32 & 40\\
\hline
63 & 28 & 49420 & 49448\\
\hline
64 & 4 & 33 & 37\\
\hline
65 & 8 & 2008 & 2016\\
\hline
66 & 2 & 34 & 36\\
\hline
67 & 8 & 87728 & 87736\\
\hline
\end{tabular}
&
\begin{tabular}{|c|c|c|c|}
\hline
$r$ & (I) & (II) & $\NA(r)$ \\
\hline
68 & 4 & 35 & 39\\
\hline
69 & 8 & 2451 & 2459\\
\hline
70 & 2 & 36 & 38\\
\hline
71 & 28 & 160306 & 160334\\
\hline
72 & 2 & 37 & 39\\
\hline
73 & 4 & 2964 & 2968\\
\hline
74 & 8 & 38 & 46\\
\hline
75 & 8 & 275919 & 275927\\
\hline
76 & 4 & 39 & 43\\
\hline
77 & 8 & 3554 & 3562\\
\hline
78 & 2 & 40 & 42\\
\hline
79 & 24 & 494159 & 494183\\
\hline
80 & 12 & 41 & 53\\
\hline
81 & 4 & 4228 & 4232\\
\hline
82 & 2 & 42 & 44\\
\hline
83 & 16 & 816756 & 816772\\
\hline
84 & 4 & 43 & 47\\
\hline
85 & 4 & 4994 & 4998\\
\hline
86 & 4 & 44 & 48\\
\hline
87 & 14 & 1450304 & 1450318\\
\hline
88 & 2 & 45 & 47\\
\hline
89 & 16 & 5860 & 5876\\
\hline
90 & 4 & 46 & 50\\
\hline
91 & 8 & 2276709 & 2276717\\
\hline
92 & 4 & 47 & 51\\
\hline
93 & 4 & 6834 & 6838\\
\hline
94 & 4 & 48 & 52\\
\hline
95 & 38 & 4116511 & 4116549\\
\hline
96 & 2 & 49 & 51\\
\hline
97 & 8 & 7925 & 7933\\
\hline
98 & 8 & 50 & 58\\
\hline
99 & 16 & 5997150 & 5997166\\
\hline
100 & 2 & 51 & 53\\
\hline
\end{tabular}
\end{tabular}

\medskip

\caption{Number of fine gradings, $\NA(r)$, on the simple Lie algebra of type $A_r$ assuming $\chr{\FF}=0$.}
\label{t:gradings_A}
\end{table}

\subsection{Asymptotic behaviour}

The function $\NA(r)$ behaves irregularly, so we use averaging similar to \eqref{defNMA}. We record for future reference:
\begin{equation}\label{defNLieA}
 \NLieA(r)=\frac{1}{r}\sum_{j\le r} \NLie(j)\quad\mbox{where}\quad X\in\{A,C,D\}.
\end{equation}
This is the average number of fine gradings on the simple Lie algebras of type $X_j$ with $j\leq r$. (The summation can start with $j=1$ if $X=A$ or $C$ and with $j=3$ if $X=D$, but such details are irrelevant for the asymptotics.)

\begin{theorem}\label{thm:asymgradings_A}
Let $\FF$ be an algebraically closed field, $\chr{\FF}\ne 2$. Let $\NAA(r)$ be defined by \eqref{defNLieA}. The following asymptotic formula holds:
\[
\ln \NAA(r)=b(r+1) r^{2/3} + O((\ln r)^2),
\]
where $b(\cdot)$ is the bounded continuous function defined by \eqref{def_beta}. Specifically, the maximum and minimum values of $b(\cdot)$ are the constants $b_0\approx 1.581080$ and $b_1\approx 1.512173$ defined in Lemma {\rm\ref{lem:sol_opt_prob}} and Equation \eqref{def_b1}.
\end{theorem}

\begin{proof}
The number of fine gradings of Type II is given by \eqref{eq:num_grad_AII}, and
Theorem \ref{thm:asymgradings_Aii} implies that the logarithm of the average
number is $b(r+1) r^{2/3} + O((\ln r)^2)$. The number of fine gradings of Type I is
given by \eqref{eq:num_grad_AI}, and Theorem \ref{thm:asymgradings_M} implies
that the average number is asymptotically negligible compared to Type II. The
result follows.
\end{proof}


\section{Lie Algebras of Series $C$}
\label{sec:C}

For Series $C$ and $D$, the classification of fine gradings (see e.g. \cite[\S 3.5, 3.6]{EKmon}) requires a certain action of $\SP_{2m}(2)$ on $T=\ZZ_2^{2m}$, which is different from the natural action (cf. \cite[p.245]{DiMo}). The group $\SP_{2m}(2)=\SP(T)$ is defined as the group of isometries of a nondegenerate alternating bilinear form on $T$, say, the following:
\begin{equation}\label{eq:bilin_form}
(x,y)=\sum_{i=1}^m x_i y_{2m+1-i}-\sum_{i=1}^m x_{2m+1-i} y_i=\sum_{i=1}^{2m} x_i y_{2m+1-i}\quad\mbox{for all }x,y\in\ZZ_2^{2m}.
\end{equation}
(This is the form used by GAP to define symplectic groups.) Now consider the quadratic form
\begin{equation}\label{eq:quad_form}
Q(x)=\sum_{i=1}^m x_i x_{2m+1-i}\quad\mbox{for all }x\in\ZZ_2^{2m}.
\end{equation}
Clearly, the bilinear form \eqref{eq:bilin_form} is the polar of $Q$, i.e., $(x,y)=Q(x+y)-Q(x)-Q(y)$ for all $x,y$. But since we are now in characteristic $2$, $Q$ cannot be expressed in terms of the bilinear form and, consequently, an element  $A\in\SP(T)$ does not necessarily preserve $Q$. One verifies that the mapping $x\mapsto Q(A^{-1}x)+Q(x)$ is linear (a peculiarity of the field of order $2$), so there exists unique $t_A\in T$ such that
\[
(t_A,x)=Q(A^{-1}x)+Q(x)\quad\mbox{for all }x\in T.
\]
It follows that $Q(Ax+t_A)=Q(x)+Q(t_A)$ for all $x$. Let
\[
T_+=\{x\in T\;|\;Q(x)=0\}\quad\mbox{and}\quad T_-=\{x\in T\;|\;Q(x)=1\}.
\]
One verifies that $|T_\pm|=2^{m-1}(2^m\pm 1)$. Since the mapping $x\mapsto Ax+t_A$ is bijective and $|T_+|\ne |T_-|$, it cannot swap $T_+$ and $T_-$, hence $Q(t_A)=0$ and $Q(Ax+t_A)=Q(x)$ for all $x$.

\begin{df}\label{df:twisted_action}
For $x\in T$ and $A\in\SP(T)$, define $A\cdot x=Ax+t_A$. One verifies that the mapping $\SP(T)\to\ASP(T)=T\rtimes\SP(T)$, $A\mapsto (t_A,A)$, is a homomorphism, so $\cdot$ is an action of $\SP(T)$ on $T$, which we call the {\em twisted action} to distinguish from the natural one.
\end{df}

By construction, the twisted action of $\SP(T)$ preserves $Q$, i.e., $T_+$ and $T_-$ are invariant subsets. One can show that the twisted action is $2$-transitive on each of $T_+$ and $T_-$.

\begin{theorem}[\cite{E09d,EK_Weyl2}]\label{C_fine}
Let $\FF$ be an algebraically closed field, $\chr{\FF}\neq 2$. Let $n\ge
4$ be even. Then any fine grading on $\Sp_n(\FF)$ is equivalent to
$\C(T,q,s,\tau)$ where $T$ is an elementary $2$-group of even rank, $q$
and $s$ are non-negative integers, $(q+2s)\sqrt{|T|}=n$,
$\tau=(t_1,\ldots,t_q)$ is a $q$-tuple of elements of $T_-$, and $t_1\ne
t_2$ if $q=2$ and $s=0$. Moreover, $\C(T_1,q_1,s_1,\tau_1)$ and
$\C(T_2,q_2,s_2,\tau_2)$ are equivalent if and only if $T_1\cong T_2$,
$q_1=q_2$, $s_1=s_2$ and, identifying $T_1=T_2=\ZZ_2^{2m}$,
$\Sigma(\tau_1)$ is conjugate to $\Sigma(\tau_2)$ by the twisted action
of $\mathrm{Sp}_{2m}(2)$ as in Definition~{\rm\ref{df:twisted_action}}.\qed
\end{theorem}

Hence, to calculate the number of fine gradings on simple Lie algebras of Series $C$, we need to determine the number of orbits, $N_-(m,q)$, of the twisted action of $\SP_{2m}(2)$ on multisets of size $q$ in $T_-\subset \ZZ_2^{2m}$. Similarly, we will need $N_+(m,q)$ for Series $D$. By the same argument as for \eqref{eq:triv_bounds}, we obtain the following bounds:
\begin{equation}\label{eq:triv_bounds_pm}
\parts_{2^{m-1}(2^m\pm 1)}(q)\le N_\pm(m,q)\le\binom{q+2^{m-1}(2^m\pm 1)-1}{q},
\end{equation}
where $\parts_k(q)$ is the number of partitions of $q$ into at most $k$ positive parts. Note that if $m=1$ then we have $|T_+|=3$ and $|T_-|=1$, so $\SP_2(2)$ acts as the full group of permutations on $T_+$ and on $T_-$, hence the lower bound is achieved in this case. It is also achieved if $q\le 2$ because of $2$-transitivity.

An alternative lower bound, similar to \eqref{eq:lower_bound_G}, is the following:
\begin{equation}\label{eq:lower_bound_G_pm}
\frac{1}{|G_m|}\binom{q+2^{m-1}(2^m\pm 1)-1}{q}\le N_\pm(m,q),
\end{equation}
where $G_m=\SP_{2m}(2)$. This will be used in Section~\ref{sec:asymACD} to obtain asymptotic results.

\subsection{Counting orbits}

Using the same method as in the previous section, we can compute the numbers $N_-(m,q)$. Some of them are displayed in Table \ref{t:orbitsSp_minus}. We note that if $m=2$ then $|T_-|=6$, so $\SP_4(2)$ acts as the full group of permutations on $T_-$ (but not on $T_+$). Hence the lower bound \eqref{eq:triv_bounds_pm} for $N_-(2,q)$ is achieved.

\begin{table}[h]
\begin{tabular}{|c|c|c|c|c|c|c|c|c|c|c|c|c|}
\hline
$q$ & 1 & 2 & 3 & 4 & 5 & 6 & 7 & 8 & 9 & 10 & 11 & 12 \\
\hline
$N_-(1,q)$ & 1 & 1 & 1 & 1 & 1 & 1 & 1 & 1 & 1 & 1 & 1 & 1\\
\hline
$N_-(2,q)$ & 1 & 2 & 3 & 5 & 7 & 11 & 14 & 20& 26 & 35 & 44 & 58 \\
\hline
$N_-(3,q)$ & 1 & 2 & 4 & 8 & 16 & 37 & 80 & 186 & 444 & 1091 & 2711 & 6857\\
\hline
$N_-(4,q)$ & 1 & 2 & 4 & 9 & 20 & 57 & 172 & 660 & 3093 & 18413 & 131556 & 1059916\\
\hline
$N_-(5,q)$ & 1 & 2 & 4 & 9 & 21 & 63 & 210 & 986 & 6773 & 77279 & 1432570 & 36967692\\
\hline
$N_-(6,q)$ &1 & 2 & 4 & 9 & 21 & 64 & 217 & 1058 & 7898 & 110027 & 3156144 & 172638169\\
\hline
\end{tabular}


\medskip

\caption{Number of orbits, $N_-(m,q)$, of $\SP_{2m}(2)$ on the set of multisets of size $q$ in $T_-\subset\ZZ_2^{2m}$.}
\label{t:orbitsSp_minus}
\end{table}

\subsection{Counting gradings}

We will use the following notation. For $n=2^\alpha k$ where $k$ is odd, set
\begin{equation}\label{eq:f_plus_minus}
f_\pm(n)=\sum_{m=0}^\alpha\sum_{s=0}^{\lfloor 2^{\alpha-m-1}k\rfloor}N_\pm(m,2^{\alpha-m}k-2s),
\end{equation}
with the convention $N_+(0,q)=1$ for all $q$ and $N_-(0,q)=\delta_{0,q}$ (Kronecker delta). Since, for Series $C$, the rank $r$ is related to the matrix size $n$ as $n=2r$, Theorem \ref{C_fine} implies that the number of fine gradings, $\NC(r)$, on the simple Lie algebra $C_r$ ($r\ge 2$) is given by
\begin{equation}\label{eq:num_grad_C}
\NC(r)=\left\{\begin{array}{ll}
                f_-(2r)-1 &\mbox{if $r$ is a power of 2},\\
                f_-(2r)   &\mbox{otherwise}.
                \end{array}\right.
\end{equation}
We calculated these numbers for $r\le 100$. The results are stated in Table \ref{t:gradings_C}, where we included the case $C_1=A_1$ for completeness. Note that if $r$ is odd, then \eqref{eq:f_plus_minus} involves only $m=0$ and $m=1$, hence $\NC(r)=\lfloor r/2 \rfloor+2$.

\begin{table}[h]
\begin{tabular}{ccccc}
\begin{tabular}{|c|c|}
\hline
$r$ & $\NC(r)$\\
\hline
1 & 2\\
\hline
2 & 3\\
\hline
3 & 3\\
\hline
4 & 7\\
\hline
5 & 4\\
\hline
6 & 9\\
\hline
7 & 5\\
\hline
8 & 17\\
\hline
9 & 6\\
\hline
10 & 18\\
\hline
11 & 7\\
\hline
12 & 32\\
\hline
13 & 8\\
\hline
14 & 34\\
\hline
15 & 9\\
\hline
16 & 63\\
\hline
17 & 10\\
\hline
18 & 62\\
\hline
19 & 11\\
\hline
20 & 107\\
\hline
\end{tabular}
&
\begin{tabular}{|c|c|}
\hline
$r$ & $\NC(r)$\\
\hline
21 & 12\\
\hline
22 & 108\\
\hline
23 & 13\\
\hline
24 & 199\\
\hline
25 & 14\\
\hline
26 & 181\\
\hline
27 & 15\\
\hline
28 & 339\\
\hline
29 & 16\\
\hline
30 & 293\\
\hline
31 & 17\\
\hline
32 & 625\\
\hline
33 & 18\\
\hline
34 & 458\\
\hline
35 & 19\\
\hline
36 & 1122\\
\hline
37 & 20\\
\hline
38 & 695\\
\hline
39 & 21\\
\hline
40 & 2211\\
\hline
\end{tabular}
&
\begin{tabular}{|c|c|}
\hline
$r$ & $\NC(r)$\\
\hline
41 & 22\\
\hline
42 & 1028\\
\hline
43 & 23\\
\hline
44 & 4510\\
\hline
45 & 24\\
\hline
46 & 1484\\
\hline
47 & 25\\
\hline
48 & 10044\\
\hline
49 & 26\\
\hline
50 & 2098\\
\hline
51 & 27\\
\hline
52 & 23038\\
\hline
53 & 28\\
\hline
54 & 2911\\
\hline
55 & 29\\
\hline
56 & 55266\\
\hline
57 & 30\\
\hline
58 & 3970\\
\hline
59 & 31\\
\hline
60 & 133241\\
\hline
\end{tabular}
&
\begin{tabular}{|c|c|}
\hline
$r$ & $\NC(r)$\\
\hline
61 & 32\\
\hline
62 & 5332\\
\hline
63 & 33\\
\hline
64 & 323502\\
\hline
65 & 34\\
\hline
66 & 7063\\
\hline
67 & 35\\
\hline
68 & 774947\\
\hline
69 & 36\\
\hline
70 & 9237\\
\hline
71 & 37\\
\hline
72 & 1838997\\
\hline
73 & 38\\
\hline
74 & 11941\\
\hline
75 & 39\\
\hline
76 & 4274302\\
\hline
77 & 40\\
\hline
78 & 15274\\
\hline
79 & 41\\
\hline
80 & 9788777\\
\hline
\end{tabular}
&
\begin{tabular}{|c|c|}
\hline
$r$ & $\NC(r)$\\
\hline
81 & 42\\
\hline
82 & 19346\\
\hline
83 & 43\\
\hline
84 & 21899478\\
\hline
85 & 44\\
\hline
86 & 24283\\
\hline
87 & 45\\
\hline
88 & 48274977\\
\hline
89 & 46\\
\hline
90 & 30227\\
\hline
91 & 47\\
\hline
92 & 103789470\\
\hline
93 & 48\\
\hline
94 & 37333\\
\hline
95 & 49\\
\hline
96 & 220645585\\
\hline
97 & 50\\
\hline
98 & 45777\\
\hline
99 & 51\\
\hline
100 & 456000882\\
\hline
\end{tabular}
\end{tabular}

\medskip

\caption{Number of fine gradings, $\NC(r)$, on the simple Lie algebra of type $C_r$ assuming $\chr{\FF}\ne 2$.}
\label{t:gradings_C}
\end{table}

\subsection{Asymptotic behaviour}

The following is an immediate consequence of \eqref{eq:num_grad_C} and Theorem
\ref{thm:asymgradings_CD}:

\begin{theorem}\label{thm:asymgradings_C}
Let $\FF$ be an algebraically closed field, $\chr{\FF}\ne 2$. Let $\NCA(r)$ be defined by \eqref{defNLieA}. The following asymptotic formula holds:
\[
\ln \NCA(r)=2^{1/3}b_-(2r) r^{2/3} + O((\ln r)^2),
\]
where $b_-(\cdot)$ is the bounded continuous function defined by \eqref{def_betapm} and described in Lemma {\rm \ref{lem:prop_bpm}}. In particular, $b_-(t)=b(t)+O(t^{-1/3})$ where $b(\cdot)$ is the function in Theorem {\rm\ref{thm:asymgradings_A}}.\qed
\end{theorem}


\section{Lie Algebras of Series $D$}
\label{sec:D}

This case is very similar to Series $C$, which was described in the previous section.

\begin{theorem}[\cite{E09d,EK_Weyl2}]\label{D_fine}
Let $\FF$ be an algebraically closed field, $\chr{\FF}\neq 2$. Let $n\ge
6$ be even. Assume $n\ne 8$. Then any fine grading on $\So_n(\FF)$ is
equivalent to $\D(T,q,s,\tau)$ where $T$ is an elementary $2$-group of
even rank, $q$ and $s$ are non-negative integers, $(q+2s)\sqrt{|T|}=n$,
$\tau=(t_1,\ldots,t_q)$ is a $q$-tuple of elements of $T_+$, and $t_1\ne
t_2$ if $q=2$ and $s=0$. Moreover, $\D(T_1,q_1,s_1,\tau_1)$ and
$\D(T_2,q_2,s_2,\tau_2)$ are equivalent if and only if $T_1\cong T_2$,
$q_1=q_2$, $s_1=s_2$ and, identifying $T_1=T_2=\ZZ_2^{2m}$,
$\Sigma(\tau_1)$ is conjugate to $\Sigma(\tau_2)$ by the twisted action
of $\mathrm{Sp}_{2m}(2)$ as in Definition~{\rm\ref{df:twisted_action}}.\qed
\end{theorem}

As mentioned in the introduction, the Lie algebra $\So_8(\FF)$ (type $D_4$) requires special treatment.

\subsection{Counting orbits}

We can compute the numbers $N_+(m,q)$ in the same way as $N_-(m,q)$ (see Table \ref{t:orbitsSp_minus}). Some of the $N_+(m,q)$ are displayed in Table \ref{t:orbitsSp_plus}.

\begin{table}[h]
\begin{tabular}{|c|c|c|c|c|c|c|c|c|c|c|c|c|}
\hline
$q$  & 1 & 2 & 3 & 4 & 5 & 6 & 7 & 8 & 9 & 10 & 11 & 12 \\
\hline
$N_+(1,q)$ & 1 & 2 & 3 & 4 & 5 & 7 & 8 & 10 & 12 & 14 & 16 & 19\\
\hline
$N_+(2,q)$ & 1 & 2 & 4 & 8 & 14 & 27 & 46 & 82& 140 & 237 & 386 & 630 \\
\hline
$N_+(3,q)$ & 1 & 2 & 4 & 9 & 20 & 53 & 138 & 408 & 1265 & 4161 & 13999 & 47628\\
\hline
$N_+(4,q)$ & 1 & 2 & 4 & 9 & 21 & 63 & 204 & 882 & 4945 & 36909 & 337821 & 3428167\\
\hline
$N_+(5,q)$ & 1 & 2 & 4 & 9 & 21 & 64 & 217 & 1048 & 7594 & 95775 & 2061395 & 62537928\\
\hline
$N_+(6,q)$ & 1 & 2 & 4 & 9 & 21 & 64 & 218 & 1067 & 8012 & 113097 & 3362409 & 198208405\\
\hline
\end{tabular}


\medskip

\caption{Number of orbits, $N_+(m,q)$, of $\SP_{2m}(2)$ on the set of multisets of size $q$ in $T_+\subset\ZZ_2^{2m}$.}
\label{t:orbitsSp_plus}
\end{table}

\subsection{Counting gradings}

Since, for Series $D$, the rank $r$ is related to the matrix size $n$ as $n=2r$, Theorem \ref{D_fine} implies that the number of fine gradings, $\ND(r)$, on the simple Lie algebra $D_r$ ($r=3$ or $r\ge 5$) is given by
\begin{equation}\label{eq:num_grad_D}
\ND(r)=\left\{\begin{array}{ll}
                f_+(2r)-1 &\mbox{if $r$ is a power of 2},\\
                f_+(2r)   &\mbox{otherwise},
                \end{array}\right.
\end{equation}
where $f_+$ is defined by \eqref{eq:f_plus_minus}. We calculated these numbers for $r\le 100$. The results are stated in Table \ref{t:gradings_D}. For completeness, we included the case $D_4$ from \cite{E09d} (where it is assumed that $\chr{\FF}=0$), for which the number of fine gradings is $17$ instead of $15$ given by the above formula. Note that if $r$ is odd, then \eqref{eq:f_plus_minus} involves only $m=0$ and $m=1$, hence $\ND(r)=\sum_{s=0}^{\lfloor r/2 \rfloor}\parts_3(1+2s)+r+1=\sum_{s=0}^{\lfloor r/2 \rfloor}\mathrm{int}\frac{(s+2)^2}{3}+r+1$, where $\mathrm{int}\,x$ denotes the integer nearest to $x$.

\begin{table}[h]
\begin{tabular}{ccccc}
\begin{tabular}{|c|c|}
\hline
$r$ & $\ND(r)$\\
\hline
& \\
\hline
& \\
\hline
3 & 8\\
\hline
4 & 17\\
\hline
5 & 15\\
\hline
6 & 26\\
\hline
7 & 25\\
\hline
8 & 47\\
\hline
9 & 39\\
\hline
10 & 68\\
\hline
11 & 57\\
\hline
12 & 113\\
\hline
13 & 80\\
\hline
14 & 161\\
\hline
15 & 109\\
\hline
16 & 263\\
\hline
17 & 144\\
\hline
18 & 372\\
\hline
19 & 186\\
\hline
20 & 595\\
\hline
\end{tabular}
&
\begin{tabular}{|c|c|}
\hline
$r$ & $\ND(r)$\\
\hline
21 & 236\\
\hline
22 & 858\\
\hline
23 & 294\\
\hline
24 & 1387\\
\hline
25 & 361\\
\hline
26 & 1987\\
\hline
27 & 438\\
\hline
28 & 3186\\
\hline
29 & 525\\
\hline
30 & 4538\\
\hline
31 & 623\\
\hline
32 & 7292\\
\hline
33 & 733\\
\hline
34 & 10069\\
\hline
35 & 855\\
\hline
36 & 16255\\
\hline
37 & 990\\
\hline
38 & 21550\\
\hline
39 & 1139\\
\hline
40 & 35756\\
\hline
\end{tabular}
&
\begin{tabular}{|c|c|}
\hline
$r$ & $\ND(r)$\\
\hline
41 & 1302\\
\hline
42 & 44335\\
\hline
43 & 1480\\
\hline
44 & 78115\\
\hline
45 & 1674\\
\hline
46 & 87671\\
\hline
47 & 1884\\
\hline
48 & 173939\\
\hline
49 & 2111\\
\hline
50 & 166968\\
\hline
51 & 2356\\
\hline
52 & 402982\\
\hline
53 & 2619\\
\hline
54 & 307013\\
\hline
55 & 2901\\
\hline
56 & 991330\\
\hline
57 & 3203\\
\hline
58 & 546543\\
\hline
59 & 3525\\
\hline
60 & 2586241\\
\hline
\end{tabular}
&
\begin{tabular}{|c|c|}
\hline
$r$ & $\ND(r)$\\
\hline
61 & 3868\\
\hline
62 & 944552\\
\hline
63 & 4233\\
\hline
64 & 7055100\\
\hline
65 & 4620\\
\hline
66 & 1588770\\
\hline
67 & 5030\\
\hline
68 & 19667958\\
\hline
69 & 5464\\
\hline
70 & 2606954\\
\hline
71 & 5922\\
\hline
72 & 54994767\\
\hline
73 & 6405\\
\hline
74 & 4181709\\
\hline
75 & 6914\\
\hline
76 & 152123321\\
\hline
77 & 7449\\
\hline
78 & 6569548\\
\hline
79 & 8011\\
\hline
80 & 413256061\\
\hline
\end{tabular}
&
\begin{tabular}{|c|c|}
\hline
$r$ & $\ND(r)$\\
\hline
81 & 8601\\
\hline
82 & 10125234\\
\hline
83 & 9219\\
\hline
84 & 1097811150\\
\hline
85 & 9866\\
\hline
86 & 15332525\\
\hline
87 & 10543\\
\hline
88 & 2848498443\\
\hline
89 & 11250\\
\hline
90 & 22842458\\
\hline
91 & 11988\\
\hline
92 & 7213746853\\
\hline
93 & 12758\\
\hline
94 & 33520718\\
\hline
95 & 13560\\
\hline
96 & 17847717516\\
\hline
97 & 14395\\
\hline
98 & 48505808\\
\hline
99 & 15264\\
\hline
100 & 43141937237\\
\hline
\end{tabular}
\end{tabular}

\medskip

\caption{Number of fine gradings, $\ND(r)$, on the simple Lie algebra of type $D_r$ assuming $\chr{\FF}\ne 2$ ($\chr{\FF}=0$ for $r=4$).}
\label{t:gradings_D}
\end{table}

\subsection{Asymptotic behaviour}

The following is an immediate consequence of \eqref{eq:num_grad_D} and Theorem
\ref{thm:asymgradings_CD}:

\begin{theorem}\label{thm:asymgradings_D}
Let $\FF$ be an algebraically closed field, $\chr{\FF}\ne 2$. Let $\NDA(r)$ be defined by \eqref{defNLieA}. The following asymptotic formula holds:
\[
\ln \NDA(r)=2^{1/3}b_+(2r) r^{2/3} + O((\ln r)^2),
\]
where $b_+(\cdot)$ is the bounded continuous function defined by \eqref{def_betapm} and described in Lemma {\rm \ref{lem:prop_bpm}}. In particular, $b_+(t)=b(t)+O(t^{-1/3})$ where $b(\cdot)$ is the function in Theorem {\rm\ref{thm:asymgradings_A}}.\qed
\end{theorem}


\section{Asymptotics of the number of fine gradings for Series $A$, $C$ and $D$}
\label{sec:asymACD}

The asymptotic formulas stated in Theorems \ref{thm:asymgradings_A},
\ref{thm:asymgradings_C} and \ref{thm:asymgradings_D} follow from similar
(rough) estimates for the functions $f_0(n)$ and $f_\pm(n)$ defined by
\eqref{eq:f_A} and \eqref{eq:f_plus_minus}, respectively. In the case of
$f_\pm(n)$, only even values of $n$ are relevant.
Define
\[
  \hat f_0(n)=\frac{1}{n}\sum_{j=1}^n f_0(j)
\]
and
\[
  \hat f_\pm(n)=\frac{1}{n/2}\sum_{2j\leq n} f_\pm(2j).
\]

Asymptotic analysis of the functions $\hat f_0$ and $\hat f_\pm$ will be
essentially based on solution of the constrained optimization problem
\begin{equation}\label{opt_prob1}
\left\{
\begin{array}{l}
u(x,y)\bydef (x+y)\ln(x+y)-x\ln x-y\ln y\,  \to\max, \\[1ex]
x>0,\; y>0,\; x^2 y=1,
\end{array}
\right.  
\end{equation}
as well as on the analysis of similar but more delicate problems (slightly different in the three cases) with certain
integrality constraints imposed on the arguments of the function $u$. The
latter problems will be dealt with in the course of the proof of
Theorems~\ref{thm:asymgradings_Aii} and \ref{thm:asymgradings_CD}. At present, let us introduce functions and
constants needed to state Theorem~\ref{thm:asymgradings_Aii}.

The problem \eqref{opt_prob1} is equivalent to maximizing the function
\begin{equation}\label{def_v}
  v(x)\bydef u(x, x^{-2}), \; x>0.
\end{equation}
The critical point equation $v'(x)=0$ can be transformed to a convenient short
form, see \eqref{eq_for_z0} below, by writing $v=x w(x^3)$, where
\[
  w(z)= z^{-1} \ln(1+z)+\ln (1+z^{-1}).
\]
Some properties of the function $u$ and the solution of 
problem~\eqref{opt_prob1} are summarized in the next lemma for reference.

\begin{lemma}\label{lem:sol_opt_prob}
(a) The function $u(x,y)$ increases in both arguments and is homogeneous of
degree 1, i.e., $u(tx,ty)=tu(x,y)$.

\smallskip
 (b) The solution of problem \eqref{opt_prob1} is
\begin{equation}\label{def_x0_y0_b}
\begin{array}{c}
  x_0={z_0}^{1/3}\approx 0.575891, \quad  y_0={z_0}^{-2/3}\approx 3.015227,\\[1ex]
  b_0\bydef u(x_0,y_0)=v(x_0)
  \approx 1.581080,
\end{array}
\end{equation}
where $z_0$ is the unique positive root of the transcendental equation
\begin{equation}\label{eq_for_z0}
   z \ln(1+z^{-1})= 2\ln(1+z).
\end{equation}\qed
\end{lemma}

The second collection of constants and functions pertains to the first ``more delicate''
optimization problem mentioned above.

Consider the transcendental equation involving the function \eqref{def_v},
\begin{equation}\label{eq:doubling_in_v}
  v(x/2)=v(x).
\end{equation}
It has a unique positive root $x_1\approx 0.800203$. Define also
\begin{equation}\label{def_b1}
  b_1\bydef v(x_1/2)=v(x_1)
\approx 1.512173.
\end{equation}

Due to \eqref{eq:doubling_in_v}, the function
\begin{equation}\label{def_v1}
  \tilde v(x) \bydef\max\{v(x),v(x/2)\}\,=\,
 \left\{\begin{array}{ll}
     v(x), \; & \mbox{\rm if $x\le x_1$} \\
     v(x/2), \; & \mbox{\rm if $x> x_1$}
  \end{array}\right.
\end{equation}
is continuous in the interval $[x_0, 2x_0]$. We will use this fact in the last
set of preliminaries, which follows.

For $t\geq 1$, let $\phi(t)$ be the multiplicative excess of $t$ over
the greatest whole power of $2$ below $t$, i.e., 
\begin{equation}\label{def_phi}
 \phi(t)=\frac{t}{2^{\lfloor \log_2 t\rfloor}}.
\end{equation}
Clearly, $1\le\phi(t)<2$ and $\phi(2t)=\phi(t)$.

Next, for $t\geq x_0^{3}$, where $x_0$ is defined in \eqref{def_x0_y0_b}, let
\begin{equation}\label{def_lambda}
 \lambda(t)=\phi(x_0^{-1} t^{1/3}),
\end{equation}
%
%
and define
\begin{equation}\label{def_beta}
 b(t)={\tilde v}(x_0\lambda(t)).
 \end{equation}
The function $b(t)$ is continuous, logarithmically periodic in the sense
that $b(8t)=b(t)$, and has bounds
\begin{equation}\label{beta_bounds}
 \min b(t)= b_1\approx 1.512173, \qquad \max b(t)=b_0\approx 1.581080.
\end{equation}
The upper bound is attained when $x_0^{-1} t^{1/3}=2^m$ with integer $m$. The
lower bound is attained when $x_1^{-1} t^{1/3}=2^m$ with integer $m$. 
The function is monotone and smooth between these maximum and minimum points.

\smallskip
We are now prepared to state the theorem describing the asymptotic behaviour of $\hat f_0(n)$. 

\begin{theorem}\label{thm:asymgradings_Aii}
There exists a constant $C>0$ such that
\begin{equation*}
 - \logsquaredterm - C\ln n \le \ln \hat f_0(n)- b(n)\,
n^{2/3} \le C\ln n,
\end{equation*}
 where the function $b(\cdot)$ is defined in \eqref{def_beta}.
 It is continuous, has property $b(8t)=b(t)$, and
 its lower and upper bounds are given in
\eqref{beta_bounds}.
\end{theorem}

\begin{proof}
Let $f^*(n)=\max_{1\le j\le n}f_0(j)$. Clearly,
\[
 \frac{1}{n}f^*(n)\le \hat f_0(n)\le f^*(n),
\]
so $\ln \hat f_0(n)=\ln f^*(n)+O(\ln n)$. Therefore, it suffices to prove the
desired estimate for $f^*(n)$ instead of $\hat f_0(n)$. (The letter $C$ will
denote a constant that may have different values in different formulas.)

Observe that in the sum \eqref{eq:f_A} defining $f_0(n)$ the number of summands
is $O(n\ln n)$. Hence, repeating the above argument, we see that
\[
\ln f_0(n) = \ln N^*+ O(\ln n),
\]
where $N^*$ is the largest summand. Now, the summands have the form $N(m,q)$,
which are the numbers of orbits as described in Section \ref{sec:A}. From
\eqref{eq:triv_bounds} and \eqref{eq:lower_bound_G} we obtain the inequalities:
\begin{equation}\label{N_via_B}
   \frac{B(m,q)}{|G_m|}\le N(m,q)\le B(m,q),
\end{equation}
where $B(m,q)$ stands for the binomial coefficient $\binom{q+2^{2m}-1}{q}$.

The required upper bound for $N^*$, and hence for $f^*(n)$, follows from the inequality
\begin{equation}\label{max_B}
\max_{2^m q\le n} \ln B(m,q) \le b(n)\, n^{2/3} +C\ln n,
\end{equation}
which will be proved in Lemma~\ref{lem:bincoef_constrained}.

To obtain the desired lower bound for $N^*$, it suffices to show that for each
$n$ there exist $m^*$ and $q^*$ with $2^{m^*} q^*\le n$ such that the following
two inequalities hold with $C$ independent of $n$:
\begin{equation}\label{ln_order_G}
\ln |G_{m^*}| \le \logsquaredterm +C\ln n,
\end{equation}
and
\begin{equation}\label{lower_bound_B}
\ln B(m^*,q^*) \ge b(n)\, n^{2/3} -C\ln n.
\end{equation}
The inequality \eqref{lower_bound_B} will be proved in Lemma~\ref{lem:bincoef_constrained}, with $m^*$ satisfying the estimate $2^{m^*}\le Cn^{1/3}$, i.e.,
\[
  m^*\le \frac{1}{3}\log_2 n + O(1).
\]
We claim that this implies \eqref{ln_order_G}. Indeed, it is well known that
\[
|\SP_{2m}(2)|=2^{m^2}\prod_{i=1}^{m}(2^{2i}-1),
\]
hence we have
\[
|G_m|= |\ZZ_2^{2m}|\cdot|\SP_{2m}(2)|\le 2^{2m}\cdot 2^{m^2+m(m+1)}=2^{2m^2+3m}.
\]
Therefore,
\[
 \ln |G_{m^*}|\le 2(m^*)^2\ln 2 +O(m^*) \le 2\ln 2\left(\frac{\log_2 n}{3}\right)^2 +O(\ln n),
\]
as claimed.
\end{proof}

\begin{lemma}\label{lem:bincoef_constrained}
Let $B(m,q)=\binom{q+2^{2m}-1}{q}$. For $t>1$, let
\begin{equation*}
  B^*(t)=\max\{B(m,q)\;|\;m,q\in\ZZ_{\ge 0}, 2^m q\le t\}.
\end{equation*}
Then $\ln B^*(t)=b(t)\, t^{2/3} +O(\ln t)$.
Moreover, there exist $q^*$ and $m^*$ such that $\ln B(m^*,q^*)=b(t)\,t^{2/3}+O(\ln t)$ and 
$q^*\le Ct^{2/3}$, $2^{m^*}\le Ct^{1/3}$.
\end{lemma}

\begin{proof}
Set $M\bydef 2^{2m}$. Let us discard the trivial case $q=0$, which clearly does
not provide maximum. So $\ln M$ and $\ln q$ are defined and nonnegative.

The condition $q\sqrt{M}\le t$ implies $\ln q\le\ln t$ and $\ln M\le 2\ln t$.
Since $|\ln\binom{q+M-1}{q}-\ln\binom{q+M}{q}|=|\ln(q+M)-\ln M| \le C\ln t$, we may replace $B(m,q)$ by $\binom{q+M}{q}$.

By Stirling's formula,
\[
 \ln\binom{q+M}{q} = u(q,M)+\frac{1}{2}\ln\frac{q+M}{qM}+O(1),
\]
where the function $u(\cdot, \cdot)$ is defined in \eqref{opt_prob1}. Again, due to the estimates $\ln q\le\ln t$ and $\ln M\le 2\ln t$, we obtain
\[
 \ln \binom{q+M}{q} = u(q,M)+O(\ln t),
\]
so we come to the optimization problem
\[
 u(q,M)\to \max, \quad q\cdot \sqrt{M} \le t,
\]
with a strong additional restriction: $q\in\ZZ_{>0}$ and $M=2^{2m}$ where $m\in\ZZ_{\ge 0}$. Without this
restriction, as Lemma~\ref{lem:sol_opt_prob} tells us, the maximum would be
equal to $b_0 t^{2/3}$ and attained at $q_0=x_0 t^{2/3}$, $M_0=y_0 t^{2/3}$. The
idea is to get our $q^*$ and $M^*$ close to these values.

Suppose first that $M$ is fixed and our only freedom is a choice of $q$, which is a nonnegative integer. To maximize $u(q,M)$ under the constraint $qM^{1/2}\le t$, we should choose the largest possible $q$, i.e., $q=\lfloor tM^{-1/2}\rfloor$. 
With this value of $q$ we have $u(q,M)=u(tM^{-1/2}, M)+O(\ln t)$. Therefore, the integrality condition on 
$q$ can be ignored.

Now we write $q=xt^{2/3}$, where $q>0$ is no longer assumed to be integer, and $M=x^{-2}t^{2/3}$. Thus we arrive at the following simplified optimization problem:
\begin{equation}\label{opt_prob2}
\left\{
\begin{array}{l}
  v(x)=u(x,x^{-2}) \to\max,  \\[2ex]
  x>0,\; x^{-1} t^{1/3}=2^m,\; m\in\ZZ_{\ge 0}.
 \end{array}
 \right.
\end{equation}
Let $\mu=\log_2(x_0^{-1}t^{1/3})$. Since the function $v(x)$ is increasing in $(0,x_0)$ and decreasing in
$(x_0,\infty)$, the optimal value of $m$ is either $\lfloor\mu\rfloor$ or $\lceil\mu\rceil$.
Looking at \eqref{def_lambda}, we see that, in the case $m=\lfloor\mu\rfloor$,
\[
 2^m=\frac{x_0^{-1}t^{1/3}}{\lambda(t)}, \quad
 x=t^{1/3} 2^{-m}=x_0\lambda(t),
\]
while in the case $m=\lceil\mu\rceil$ ($\mu\notin\ZZ$),
\[
 2^{m-1}=\frac{x_0^{-1}t^{1/3}}{\lambda(t)}, \quad
 x=t^{1/3} 2^{-m}=\frac{1}{2}x_0\lambda(t).
\]
Denote for a moment $x(t)=x_0\lambda(t)$. 
Then we have $x=x(t)$ in the first case and $x=x(t)/2$ in the second case. Recalling the definition \eqref{def_b1} of $b_1$, we see that the inequality $v(x(t))\ge v(x(t)/2)$ holds if and only if $x(t)\le x_1$. In view of \eqref{def_v1} and \eqref{def_beta}, we conclude that the solution of the optimization problem \eqref{opt_prob2} is exactly $b(t)$. Hence
\[
 \max\{u(q,2^{2m})\;|\;q>0,m\in\ZZ_{\ge 0},\, q\cdot 2^m\le t\} = b(t)\, t^{2/3}.
\]
The claimed asymptotics of $\ln B^*(t)$ follows. Finally, note that the values $m^*=\lfloor\mu\rfloor$ or $\lceil\mu\rceil$ (chosen as explained above) and $q^*=\lfloor x(t)t^{2/3}\rfloor$ or $\lfloor \frac{x(t)}{2}t^{2/3}\rfloor$ (respectively) satisfy the required conditions.
\end{proof}

To state Theorem~\ref{thm:asymgradings_CD} we need two more function, $b_\pm(t)$, which will play the role of $b(t)$ in Theorem~\ref{thm:asymgradings_Aii}. We begin with substitutes for the function $v(x)$ defined by \eqref{def_v}, which are
\begin{equation}\label{def_v_tau}
  v^\pm_\tau(x)=u(x,x^{-2}\pm \tau x^{-1}).
\end{equation} 
Here $x,\tau>0$; in the case of $v^-_\tau$ we also assume that $x^{-2}-\tau x^{-1}>0$. 

By the Implicit Function Theorem (IFT), for sufficiently small $\tau$ the transcendental equation 
\begin{equation}\label{def_x_tau}
 \frac{d}{dx} v^\pm_\tau(x)= 0
\end{equation} 
has a unique root near the root $x_0$ of the equation $v'(x)=0$.  
Denote that root $x^{\pm}_{0,\tau}$.
For sufficiently small $\tau$, the function $v^\pm(\tau)$ attains its maximum at $x^\pm_{0,\tau}$.

Similarly to \eqref{def_v1}, \eqref{def_lambda} and \eqref{def_beta}, we define
\begin{align}
  \tilde v^\pm_\tau(x)&=\max\{v^\pm_\tau(x),\,v^\pm_\tau(x/2)\},\label{def_v1_tau}\\
  \lambda^\pm_\tau(t)&=\phi\left(\frac{t^{1/3}}{x^\pm_{0,\tau}}\right),\label{def_lambda_tau}\\
  b_\pm(t)&=\left.\tilde v^\pm_\tau\left(x^\pm_{0,\tau} \cdot
\lambda^\pm_\tau(t)
\right)\right|_{\tau=(2t)^{-1/3}},\label{def_betapm}
\end{align} 
where $\phi(\cdot)$ is defined in \eqref{def_phi}. (Note that $b_\pm(t)$ are defined for sufficiently large $t$.)

The behaviour of the functions $b_\pm(t)$ with small $\tau$ is similar to that of $b(t)$. In particular, they are positive, bounded, and separated from zero (see Lemma~\ref{lem:prop_bpm}, below, for more precise information).   

\begin{theorem}\label{thm:asymgradings_CD}
There exists a constant $C>0$ such that
\begin{equation}\label{asymgradings_CD}
 - \logsquaredterm - C\ln n \le \ln \hat f_\pm(n) - 2^{-1/3}b_\pm (n)\,
n^{2/3} \le C\ln n,
\end{equation}
where the functions $b_\pm(\cdot)$ are defined in \eqref{def_betapm}.
\end{theorem}

\begin{proof}
We follow the proof of Theorem~\ref{thm:asymgradings_Aii} with minor modifications, so 
we only describe the changes that need to be made. The cases of $f_+(n)$ and $f_-(n)$ are completely analogous;
their only difference is the sign $+$ or $-$ in various formulas. 

In the estimate \eqref{N_via_B}, the group is now $G_m=\SP_2(m)$ and $B(m,q)$ stands for the binomial coefficient $\binom{q+2^{m-1}(2^m\pm1)-1}{q}$. Since $|\SP_2(m)|<|\ASP_2(m)|$, the inequality \eqref{ln_order_G} is proved as before.  
It remains to apply the next lemma, which is an adaptation of Lemma~\ref{lem:bincoef_constrained}.
\end{proof}

\begin{lemma}\label{lem:bincoef_constrained2}
Let $B(m,q)=\binom{q+2^{m-1}(2^{m}\pm 1)-1}{q}$. For $t>1$, let
\begin{equation*}
  B^*(t)=\max\{B(m,q)\;|\;m,q\in\ZZ_{\ge 0}, 2^m q\le t\}.
\end{equation*}
Then $\ln B^*(t)=2^{-1/3} b_\pm(t)\, t^{2/3} +O(\ln t)$.
Moreover, there exist $q^*$ and $m^*$ such that $\ln B(m^*,q^*)=2^{-1/3} b_\pm(t)\,t^{2/3}+O(\ln t)$ and 
$q^*\le Ct^{2/3}$, $2^{m^*}\le Ct^{1/3}$.
\end{lemma}

\begin{proof}
As in the proof of Lemma~\ref{lem:bincoef_constrained},
maximization of $B^*(t)$, to the accuracy of $O(\ln t)$, reduces to maximization of $u(q,M)$
under the constraints
\begin{equation}\label{opt_prob_CD}  
2M=2^m(2^{m}\pm 1), \quad q\cdot 2^m=t, \quad m\in\ZZ_{\ge 0};
\end{equation}
the condition $q\in\ZZ$ is dropped here.

Letting $q=2^{-1/3} x\cdot t^{2/3}$, we have $2^m=(2t)^{1/3}x^{-1}$,
\[
  M=2^{-1/3} x^{-2} t^{2/3} (1\pm(2t)^{-1/3} x),
\]
and
\[
  u(q,M)= \left. 2^{-1/3} t^{2/3}\, v^\pm_\tau(x)\right|_{\tau=(2t)^{-1/3}}.
\]
The rest of the proof repeats that of Lemma~\ref{lem:bincoef_constrained}. 
\end{proof}

The functions $b_\pm(t)$ used in Theorem~\ref{thm:asymgradings_CD} can be computed by \eqref{def_betapm} and
preceding formulas. However, understanding of their qualitative behaviour is obscured by the involvement of the parameter $\tau$.  
It is desirable to have simpler even if approximate expressions for $b_\pm(t)$ in terms of functions of just one variable. 
This is the purpose of Lemma~\ref{lem:prop_bpm}. Informally it says that for large $t$, on every interval between $t$ and $8t$
with deleted subinterval of size $O(t^{-1/3})$ the functions $b_\pm(t)$ have two-term asymptotics \eqref{asym_b_pm},
while if the exceptional subintervals are not deleted, then there is a uniform but less precise approximation \eqref{b1_bounds}.

If we replace $b_\pm(t)$ in Theorem~\ref{thm:asymgradings_CD} by their asymptotics \eqref{asym_b_pm}, the error terms
$O(t^{-2/3})$ give rise to the error of order $O(1)$ in the middle part of the inequality \eqref{asymgradings_CD},
which can be discarded at the expense of a possible increase of the constant $C$. 

The exceptional intervals are $O(t^{-1/3})$-neighbourhoods of values of $t$ corresponding to the switching point of maximum in 
\eqref{def_v1}. In those intervals, $b_\pm(t)=b_1+O(t^{-1/3})$ with $b_1$ as in \eqref{def_b1} and \eqref{beta_bounds}. For the corresponding
values of $n$ in \eqref{asymgradings_CD}, our simplified estimate \eqref{b1_bounds} yields a coarser asymptotics,
\[
 \hat f_\pm(n) = 2^{-1/3}b_1 n^{2/3} + O(n^{1/3}).
\]

\begin{lemma}\label{lem:prop_bpm}
(i) The functions $b_\pm(t)$ are continuous and bounded. Moreover, 
\begin{equation}\label{b1_bounds}
   b_\pm(t)= b(t)+O(t^{-1/3}),  
\end{equation}
where $b(t)$ is defined by \eqref{def_beta}.

(ii) Recall $x_0$ defined in Lemma \ref{lem:sol_opt_prob} and $x_1$ defined by \eqref{eq:doubling_in_v}. 
There exists $C>0$ such that for all sufficiently large $t$ satisfying
\begin{equation}\label{reg_interval}
\left|\phi\left(\frac{t^{1/3}}{x_0}\right)-\frac{x_1}{x_0}\right| > Ct^{-1/3}
\end{equation}
we have
\begin{equation}\label{asym_b_pm}
 b_\pm(t)= b(t)\pm t^{-1/3} b^{(1)}(t) +O(t^{-2/3}),
\end{equation}
where $b^{(1)}(t)$ is a bounded function given by the formula of a type similar to \eqref{def_beta}:
\begin{equation}\label{expr_b1}
  b^{(1)}(t)=\frac{\ln(1+x^3(t))}{x(t)}, 
\qquad
  x(t)=\left\{\begin{array}{ll} x_0 \lambda(t)\; & \mbox{\rm if $x_0 \lambda(t)<x_1$},
  \\   (x_0/2) \lambda(t)\; & \mbox{\rm  if $x_0 \lambda(t)>x_1$}.
 \end{array}\right. 
\end{equation}
In particular, 
$b^{(1)}(8t)=b^{(1)}(t)$. 
\end{lemma}

\begin{proof}
Let us analyse the function $b_+(\cdot)$, say. To lighten notation,  we will write $v_\tau$ instead of $v^+_\tau$, etc.,
when referring to the functions \eqref{def_v_tau}--\eqref{def_lambda_tau}. Whenever the behaviour of a function with subscript $\tau$ with respect to $t$ is discussed, it will be assumed that $\tau=(2t)^{-1/3}$ unless stated otherwise.  

\smallskip
(i) First, note the asymptotics of the critical point $x_{0,\tau}$ of the function $v_\tau(x)$:
$x_{0,\tau}=x_0+O(\tau)$ (by IFT).

The argument of the function $\tilde v_\tau$ in \eqref{def_betapm} lies in 
%
$[x_{0,\tau}, 2x_{0,\tau}]\subset I_0\bydef [x_0/4, 4x_0]$ for sufficiently large $t$.
Regarding $v_\tau(x)$ as a small perturbation of the function $v(x)=u(x,x^{-2})$, we have the approximation
\[
  v_\tau(x)=v(x)+ O(\tau)\quad 
\mbox{\rm as $\tau\to 0$} 
\]
with uniform remainder term for $x\in I_0$. The same approximation holds true with $x$ replaced by $x/2$
and consequently for $\tilde v_\tau(x)$ and $\tilde v(x)$ instead of, respectively, $v_\tau(t)$ and $v(x)$,
cf.\ \eqref{def_v1} and \eqref{def_v1_tau}.
Thus we obtain the asymptotics \eqref{b1_bounds}. 

The function $b_+(t)$ is certainly continuous at all points of continuity of $\lambda_\tau(t)$. In order to check its continuity
for all 
values of $t$, it remains to consider the points where $\lambda_\tau(t)$ has jump discontinuities.
If $t=\theta$ is such a point, then there are two limit values of $x_{0,\tau} \lambda_\tau(t)$ as $t\to\theta$, namely, $x_{0,\tau}$ and $2x_{0,\tau}$. Hence, the limit values of $b_+(t)$ are $\tilde v_\tau(x_{0,\tau})$ and $\tilde v_\tau(2x_{0,\tau})$. But these are equal by definition \eqref{def_v1_tau} of $\tilde v_\tau$.     

\smallskip
(ii) 
We introduce two auxiliary functions of the variables $t$ and $\tau$, temporarily considered as independent:
\begin{equation}\label{beta_k}
  \beta_k(t,\tau)=v_\tau\left(\frac{x_{0,\tau}}{k}\cdot \phi\left(\frac{t^{1/3}}{x_{0,\tau}}\right)\right), \quad k\in\{1,2\}.
\end{equation}
Then 
\[
b_+(t)=\max_{k}\beta_k(t, (2t)^{-1/3}).
\]
Let us analyse more carefully which of the two values of $k$ provides the maximum and where the switch between $k=1$ and $k=2$ occurs.

By IFT, for sufficiently small values of $\tau$ the equation $v_\tau(x/2)=v_\tau(x)$ has a unique root $x_{1,\tau}$ 
in the interval $[x_0, 2x_0]$, and $x_{1,\tau}=x_1+O(\tau)$.
(We omit a verification of applicability of IFT, which amounts to numerical evaluation with guaranteed maximum error.)   

In the interval $x\in [x_{0,\tau}/2, 2x_{0,\tau}]$ the function $v_\tau(x)$ has maximum at $x_{0,\tau}=x_0+O(\tau)$; 
it is increasing for $x<x_{0,\tau}$ and decreasing for $x>x_{0,\tau}$.  
Hence, to obtain $b_+(t)$, we must substitute $\tau=(2t)^{-1/3}$ and take $k=1$ if $\phi(t^{1/3}/x_{0,\tau})\le x_{1,\tau}/x_{0,\tau}$ and $k=2$ otherwise.

\smallskip
We claim that the condition 
\begin{equation*}
   (-1)^k \left(\phi\left(\frac{t^{1/3}}{x_0}\right)-\frac{x_1}{x_0}\right) > C\tau 
\end{equation*}
with sufficiently large $C$ implies 
\begin{equation*}
   (-1)^k \left(\phi\left(\frac{t^{1/3}}{x_{0,\tau}}\right)-\frac{x_{1,\tau}}{x_{0,\tau}}\right) > 0.
\end{equation*}
It is clearly so if $\phi(x_0^{-1}t^{1/3})$ is close to the extreme values 1 or 2, so assume that $\phi(x_0^{-1}t^{1/3})\in (1+\eps,2-\eps)$ with 
some $\eps>0$. To be specific, suppose $(-1)^k=1$ and write the implication to prove in the abridged form 
$(\phi_0-\xi_0>C\tau) \,\stackrel{?}{\Rightarrow}\, (\phi_\tau-\xi_\tau>0)$. 
By a standard argument with triangle inequality, it suffices to show that $|\phi_0-\phi_\tau|<C\tau/2$ and $|\xi_0-\xi_\tau|<C\tau/2$. 
The latter inequality with appropriate $C$ is equivalent to the obvious estimate $(x_{1,\tau}/x_{0,\tau}) - (x_1/x_0)=O(\tau)$.
As for the former one, note that $\phi(\cdot)$ is linear homogeneous on its intervals of continuity, so
\[
\frac{\phi_0}{\phi_\tau}=\frac{x_{0,\tau}}{x_0} =1+O(\tau).  
\]
The estimate $\phi_0-\phi_\tau=O(\tau)$ follows.

\smallskip
The rest is simple: once $k$ is determined by the value of $\phi(x_0^{-1}t^{1/3})$, we write
\[
  \beta_k(t,\tau)=\beta_k(t,0) +\left.\frac{\partial\beta_k}{\partial\tau}\right|_{\tau=0}\,\tau +O(\tau^2).
\]
Here $\beta_k(t,0)=b(t)$ and the second term has the form $b^{(1)}(t)\tau$. Since $\partial_\tau\beta(t,0)$ depends on $t$
only through the function $\phi(x_0^{-1}t^{1/3})$ in the argument of $v'(\cdot)$, it is clear that $b^{(1)}(t)$ is
bounded and $b^{(1)}(t)=b^{(1)}(8t)$. 

It remains to complete calculation of $b^{(1)}(t)$. It will follow that the result in the case of $b_-(t)$ differs from
that in the case of $b_+(t)$ just in sign. 
 
Observe that the argument of the function $v_\tau$ in \eqref{beta_k}
is locally constant, hence $(\partial/\partial\tau)\beta_k(t,\tau)=(\partial v_\tau/\partial\tau)(\dots)$.
Looking at the definition \eqref{def_v_tau} of the functions $v^{\pm}_\tau$ and recalling the definition of $u$ in
\eqref{opt_prob1}, we find 
\[
 \left.\frac{\partial v^{\pm}_\tau}{\partial\tau}\right|_{\tau=0}= \pm x^{-1}\left.\frac{\partial u(x,y)}{\partial y}\right|_{y=x^{-2}}
= \pm x^{-1} \ln(x^3+1).
\]
Taking $k=1$ or $2$ according to the sign of $x_{0,\tau}\phi(t^{1/3}/x_{0,\tau})-x_1$, we come to the formula \eqref{expr_b1}.
\end{proof}

\section*{Acknowledgments}

The authors are grateful to John Irving, Jonny Lomond, Attila Mar\'oti, Csaba Schneider and Nabil Shalaby for useful discussions. Our special thanks are due to Leandro Vendramin for help with GAP.



\begin{thebibliography}{ABCXX}
\bibitem[BSZ01]{BSZ}
Bahturin, Y.; Sehgal, S.; Zaicev, M. {\it Group gradings on associative algebras}. J. Algebra \textbf{241} (2001), no.~2, 677--698.
\bibitem[BZ03]{BZ03} Bahturin, Y. and Zaicev, M. {\it Graded algebras and graded identities}.
Polynomial identities and combinatorial methods (Pantelleria, 2001), 101--139,
Lecture Notes in Pure and Appl. Math., \textbf{235}, Dekker, New York,  2003.
\bibitem[DM96]{DiMo} Dixon, J. D. and Mortimer, B. {\it Permutation groups}. Graduate Texts in Mathematics, \textbf{163}, Springer-Verlag, New York, 1996.
\bibitem[DM06]{DM_g2} Draper, C. and Mart{\'\i}n, C. {\it Gradings on $G_2$}. Linear Algebra Appl. {\bf 418} (2006), no.~1, 85--111.
\bibitem[DM09]{DM_f4} Draper, C. and Mart{\'\i}n, C. {\it Gradings on the Albert Algebra and on $F_4$}. Rev. Mat. Iberoam. \textbf{25} (2009), no.~3, 841--908.
\bibitem[DV]{DV_e6} Draper, C. and Viruel, A. {\it Fine gradings on $E_6$}, preprint arXiv: 1207.6690 [math.RA]
\bibitem[Eld10]{E09d} Elduque, A. {\it Fine gradings on simple classical Lie algebras.} J. Algebra {\bf 324} (2010), no. 12, 3532--3571.
\bibitem[EK12a]{EK_g2f4} Elduque, A. and Kochetov, M. {\it Gradings on the exceptional Lie algebras $F_4$ and $G_2$ revisited.} Rev. Mat. Iberoam. {\bf 28} (2012), no.~3, 773--813.
\bibitem[EK12b]{EK_Weyl1} Elduque, A. and Kochetov, M. {\it Weyl groups of fine gradings on matrix algebras, octonions and the Albert algebra.} J. Algebra {\bf 366} (2012), 165--186.
\bibitem[EK12c]{EK_Weyl2} Elduque, A. and Kochetov, M. {\it Weyl groups of fine gradings on simple Lie algebras of types $A$, $B$, $C$ and $D$.} Serdica Math. J. {\bf 38} (2012), 7--36.
\bibitem[EK13]{EKmon} Elduque, A. and Kochetov, M. \emph{Gradings on simple {L}ie algebras}. Mathematical Surveys and Monographs \textbf{189},  American Mathematical Society, Providence, RI, 2013.
\bibitem[ES35]{ErdSzek} Erd\"os, P. and Szekeres, G. {\it \"Uber die Anzahl der Abelschen Gruppen gegebener Ordnung und \"uber ein verwandtes zahlentheoretisches Problem.} Acta Scient. Math. Szeged {\bf 7} (1935), 95--102.
\bibitem[GAP]{GAP4} The GAP~Group, \emph{GAP -- Groups, Algorithms, and Programming, Version 4.5.6}; 2012 (\texttt{http://www.gap-system.org}).
\bibitem[HPP98a]{HPP98} Havl\'{\i}cek, M.; Patera, J.; Pelantov\'{a}, E. {\it On {L}ie gradings. {II}}. Linear Algebra Appl. \textbf{277} (1998), no.~1-3, 97--125.
\bibitem[HPP98b]{HPP98sl3} Havl\'{\i}cek, M.; Patera, J.; Pelantov\'{a}, E. {\it Fine gradings of the real forms of $\mathfrak{sl}(3,\CC)$}.
(Russian) ; translated from  Yadernaya Fiz.  61  (1998),  no. 12, 2297--2300 Phys. Atomic Nuclei  {\bf 61}  (1998),  no. 12, 2183--2186.
\bibitem[HB89]{HB} Heath-Brown, D. R. {\it The number of abelian groups of order at most $x$.} Journ\'ees Arithm\'etiques, 1989 (Luminy, 1989). Ast\'erisque  \textbf{198-200} (1991), 153--163.
\bibitem[Koc09]{Ksur} Kochetov, M. {\it Gradings on finite-dimensional simple Lie algebras.} Acta Appl. Math. \textbf{108} (2009), no.~1, 101--127.
\bibitem[Liu91]{Liu} Liu, Hong-Quan. {\it On the number of abelian groups of a given order.} Acta Arith. \textbf{59} (1991), 261--277.
\bibitem[PPS01]{PPS01} Patera, J.; Pelantov\'{a}, E.; Svobodov\'{a}, M. {\it Fine gradings of $\mathfrak{o}(5,\CC)$, $\mathfrak{sp}(4,\CC)$ and of their real forms.} J. Math. Phys.  {\bf 42}  (2001),  no. 8, 3839--3853.
\bibitem[PPS02]{PPS02} Patera, J.; Pelantov\'{a}, E.; Svobodov\'{a}, M. {\it The eight fine gradings of $\mathfrak{sl}(4,\CC)$ and $\mathfrak{o}(6,\CC)$.}
 J. Math. Phys.  {\bf 43}  (2002),  no. 12, 6353--6378.
\bibitem[Svo08]{Svo08} Svobodov\'{a}, M. {\it Fine gradings of low-rank complex Lie algebras and of their real forms.}
 SIGMA Symmetry Integrability Geom. Methods Appl.  {\bf 4}  (2008), Paper 039, 13 pp.
\end{thebibliography}
\end{document}